%% file: LowRankKrylovMethods_newversion.tex
\documentclass[10pt]{article}
\usepackage{amsmath,amsfonts,amssymb,graphicx}
\usepackage{mathtools}
\usepackage{nccmath}
\makeatletter
\newcommand*\rel@kern[1]{\kern#1\dimexpr\macc@kerna}
\newcommand*\widebar[1]{%
  \begingroup
  \def\mathaccent##1##2{%
    \rel@kern{0.8}%
    \overline{\rel@kern{-0.8}\macc@nucleus\rel@kern{0.2}}%
    \rel@kern{-0.2}%
  }%
  \macc@depth\@ne
  \let\math@bgroup\@empty \let\math@egroup\macc@set@skewchar
  \mathsurround\z@ \frozen@everymath{\mathgroup\macc@group\relax}%
  \macc@set@skewchar\relax
  \let\mathaccentV\macc@nested@a
  \macc@nested@a\relax111{#1}%
  \endgroup
}
\makeatother
\DeclareMathOperator*{\argmin}{argmin}
\usepackage{color}
\usepackage{amsthm}
\usepackage{algorithmic,algorithm}
\usepackage[ruled,algosection,algo2e,resetcount]{algorithm2e}

\input{generic}
\usepackage{pgfplots}			%	matlab2tikz 
\pgfplotsset{compat=1.3}		%	double xlabel
 \usepackage{multirow}
\usepackage{scalerel,stackengine}
\stackMath
\newcommand\reallywidehat[1]{%
\savestack{\tmpbox}{\stretchto{%
  \scaleto{%
    \scalerel*[\widthof{\ensuremath{#1}}]{\kern-.6pt\bigwedge\kern-.6pt}%
    {\rule[-\textheight/2]{1ex}{\textheight}}%WIDTH-LIMITED BIG WEDGE
  }{\textheight}% 
}{0.5ex}}%
\stackon[1pt]{#1}{\tmpbox}%
}

\newcommand{\RR}{{\mathbb{R}}}

\usepackage{geometry}
\usepackage{todonotes}

\newtheorem{Prop}{Proposition}[section]
\newtheorem{theorem}{Theorem}[section]
\newtheorem{remark}[theorem]{Remark}
\newtheorem{num_example}[theorem]{Example}

\begin{document}

\bibliographystyle{siam}

\title{On the convergence of Krylov methods with low-rank truncations}

\author{
Davide Palitta\thanks{\texttt{palitta@mpi-magdeburg.mpg.de}, Research Group Computational Methods in Systems
and Control Theory (CSC),
Max Planck Institute for Dynamics of Complex Technical Systems,
Sandtorstra\ss{e} 1, 39106 Magdeburg, Germany} 
\and Patrick K\"{u}rschner\thanks{\texttt{patrick.kurschner@kuleuven.be}
%{\sc
Department of Electrical Engineering  (ESAT), ESAT/STADIUS, KU Leuven, 
Kasteelpark Arenberg 10, 3001 Leuven, Belgium}
}%}

\maketitle

\begin{abstract}{
Low-rank Krylov methods are one of the few options available in the literature to address the numerical solution of large-scale general linear matrix equations. These routines amount to well-known Krylov schemes that have been equipped with a couple of low-rank truncations to maintain a feasible storage demand in the overall solution procedure. However, such truncations may affect the convergence properties of the adopted Krylov method. In this paper we show how the truncation steps have to be performed in order to maintain the convergence of the Krylov routine. Several numerical experiments validate our theoretical findings.}

\end{abstract}
% 
% \begin{keywords}
% Linear matrix equations, Krylov subspace methods, low-rank methods, low-rank truncations.
% \end{keywords}
% 
% %\begin{AMS}
% %65F10, 65F30, 15A06, 15A24
% %\end{AMS}
% 
% %%%%%%%%%%%%%%%%%%%%%%%%%%%%%%%%%%%%%%%%%%%%%%%%%%%%%%%%%%%%%%%%%%%%%%%%%%%%%%%%%%%%%%%%%%%%%%%%%%%%%%%%%%%%%%%%%%%%%%%

\section{Introduction}

We are interested in the numerical solution of general linear matrix equations of the form
\begin{equation}\label{eq_main}
 \sum_{i=1}^pA_iXB_i^T+C_1C_2^T=0,
\end{equation}
where $A_i\in\mathbb{R}^{n_A\times n_A}$, $B_i\in\mathbb{R}^{n_B\times n_B}$ are large matrices
that allow  matrix-vector products $A_iv$, $B_iw$ to be efficiently computed for all $i=1,\ldots,p$, and any $v\in\mathbb{R}^{n_A}$, $w\in\mathbb{R}^{n_B}$.
Moreover, $C_1$, $C_2$ are supposed to be of low rank, i.e., $C_1\in\mathbb{R}^{n_A\times q}$, $C_2\in\mathbb{R}^{n_B\times q}$, $q\ll n_A,n_B$. For sake of simplicity we consider the case of $n_A=n_B\equiv n$ in the following, so that the solution $X\in\mathbb{R}^{n\times n}$ is a square matrix, but our analysis can be applied to the rectangular case, with $n_A\neq n_B$, as well.

Many common linear matrix equations can be written as in~\eqref{eq_main}. For instance, if $p=2$ and $B_1=A_2=I_n$, $I_n$  identity matrix of order $n$, we get the classical Sylvester equations. %the (standard) Sylvester equation if $B_1=A_2=I_n$ where $I_n$ denotes the identity matrix of order $n$. If in addition $B_2=A_1$, $A_2=B_1$ and $C_1=C_2$, the generalized 
Moreover, if $B_2=A_1$, $A_2=B_1$, and $C_1=C_2$, the Lyapunov equation is attained.
 %is obtained and its standard counterpart can be attained if $A_2=B_1=I_n$. 
These equations are ubiquitous in signal processing and control and systems theory. See, e.g.,~\cite{Antoulas.05,Dooren1991,Benner2017}. The discretization of certain elliptic PDEs yields Lyapunov and Sylvester equations as well. See, e.g.,~\cite{Palitta2016,Breiten2016}.

%Another important kind of linear matrix equation can be derived from~\eqref{eq_main}. Even though they have a different nature with respect to the aforementioned ones, also these equations are sometimes referred to as 
\emph{Generalized} Lyapunov and Sylvester equations\footnote{We note that also for $p=2$, the equations we get when $B_1\neq I_n$, $A_2\neq I_n$ are sometimes referred to as generalized Sylvester (Lyapunov) equations. In this work the term \emph{generalized} always refers to the case $p >2$  consisting of a Lyapunov/Sylvester operator plus a linear operator.} amount to a Lyapunov/Sylvester operator plus a general linear operator:
%They can be thus written in the form 
%%
$$ AXB^T+BXA^T+\sum_{i=1}^{p-2}N_iXN_i^T+CC^T=0,\quad\text{and}\quad A_1XB_1+A_2XB_2^T+\sum_{i=1}^{p-2}N_iXM_i^T+C_1C_2^T=0.
$$
See, e.g.,~\cite{Benner2013,Jarlebring2018}. These equations play an important role in model order reduction of bilinear and stochastic systems, see, e.g., \cite{Benner2013, Benner2011, Damm2008}, and many problems arising from the discretization of PDEs can be formulated as generalized Sylvester equations as well.
See, e.g., \cite{Palitta2016, Ringh2018,WeiBR19}.

General multiterm linear matrix equation of the form~\eqref{eq_main} have been attracting attention in the very recent literature because they arise in many applications like the discretization of deterministic and stochastic PDEs, see, e.g.,  \cite{Baumann2018,Powell2017}, PDE-constrained optimization problems~\cite{Stoll2015}, data assimilation~\cite{Freitag2018}, matrix  regression problems arising in computational neuroscience~\cite{Kuerschner2018}, fluid-structure interaction problems \cite{WeiBR19}, and many more.

Even when the coefficient matrices $A_i$'s and $B_i$'s in~\eqref{eq_main} are sparse, the solution $X$ is, in general, dense and it cannot be stored for large scale problems.
However, for particular instances of~\eqref{eq_main}, as the ones above, and under certain assumptions on the coefficient matrices,  
 a fast decay in the singular values of $X$ can be proved and, thus, the solution admits accurate low-rank approximations of the form $S_1S_2^T\approx X$, $S_1$, $S_2\in\mathbb{R}^{n\times t}$, $t\ll n$, so that only the low-rank factors $S_1$ and $S_2$ need to be computed and stored. See, e.g.,~\cite{Penzl2000,Baker2015,Benner2013,Jarlebring2018}.  
 
 For the general multiterm linear equation~\eqref{eq_main}, robust low-rank approximability properties of the solution have not been established so far even though $X$ turns out to be numerically low-rank in many cases. See, e.g.,~\cite{Stoll2015, Freitag2018}. In the rest of the paper we thus assume that the solution $X$ to~\eqref{eq_main} admits accurate low-rank approximations.
 
 The efficient computation of the low-rank factors $S_1$ and $S_2$ is the task of the so-called low-rank methods and many different algorithms
have been developed in the last decade for both generalized and standard Lyapunov and Sylvester equations. A non complete list of low-rank methods for such equations includes projection methods proposed in, e.g.,~\cite{Druskin.Simoncini.11,Simoncini2007,Jarlebring2018,Shank2015,Powell2017}, low-rank (bilinear) ADI iterations~\cite{Benner2009,Li2004, Benner2013},
sign function methods~\cite{Baur2006,Baur2008}, and Riemannian optimization methods~\cite{Kressner2016,Vandereycken2010}. We refer the reader to~\cite{Simoncini2016} for a thorough presentation of low-rank techniques.

To the best of our knowledge, few options are present in the literature for the efficient numerical solution of general equations~\eqref{eq_main}: A greedy low-rank method by Kressner and Sirkovi\'{c}~\cite{Kressner2015}, and low-rank Krylov procedures (e.g.,~\cite{Kressner2011,Stoll2015,Freitag2018,Benner2013}) which are the focus of this paper. 

Krylov methods for matrix equations can be seen as standard Krylov subspace schemes applied to the $n^2\times n^2$ linear system
\begin{equation}\label{eq:linear_system}
 \mathcal{A}\text{vec}(X)=-\text{vec}(C_1C_2^T), \quad \mathcal{A}:=\left(\sum_{i=1}^p B_i\otimes A_i\right)\in\mathbb{R}^{n^2\times n^2},
\end{equation}
where $\otimes$ denotes the Kronecker product and $\text{vec}:\mathbb{R}^{n\times n}\rightarrow \mathbb{R}^{n^2}$ is such that $\text{vec}(X)$ is the vector obtained by stacking the columns of
the matrix $X$ one on top of each other. 

These methods construct the Krylov subspace
\begin{equation}\label{eq.DefKrylovSpace}
 \mathbf{K}_m(\mathcal{A},\text{vec}(C_1C_2^T))=\text{span}\left\{\text{vec}(C_1C_2^T),
 \mathcal{A}\text{vec}(C_1C_2^T),\ldots,\mathcal{A}^{m-1}\text{vec}(C_1C_2^T)\right\},
\end{equation}
and compute an approximate solution of the form 
$\text{vec}(X_m)=V_my_m\approx \text{vec}(X)$, where $V_m=[v_1,\ldots,v_m]\in\mathbb{R}^{n^2\times m}$ has orthonormal columns and it is such that 
$\text{Range}(V_m)=\mathbf{K}_m(\mathcal{A},\text{vec}(C_1C_2^T))$ with $y_m\in\mathbb{R}^m$. The vector $y_m$ can be computed in different ways which depend on the selected Krylov method. The most common schemes are based either on a (Petrov-)Galerkin 
condition on the residual vector or a minimization 
procedure of the residual norm; see, e.g.,~\cite{Saad2003}.

The coefficient matrix $\mathcal{A}$ in~\eqref{eq:linear_system} is never assembled explicitly in the construction of $\mathbf{K}_m(\mathcal{A},\text{vec}(C_1C_2^T))$ but its
Kronecker structure is exploited to efficiently perform matrix-vector products. Moreover, to keep the memory demand low, the basis vectors of $\mathbf{K}_m(\mathcal{A},\text{vec}(C_1C_2^T))$ must 
be stored in low-rank format. To this end, the Arnoldi procedure to compute $V_m$ has to be equipped with a couple of low-rank truncation steps. In particular, a low-rank truncation is performed after
the ``matrix-vector product'' $\mathcal{A}v_m$ where $v_m$ denotes the last basis vector, and during the orthogonalization process. See, e.g.,~\cite[Section 3]{Stoll2015}, \cite[Section 2]{Kressner2011},~\cite[Section 3]{Freitag2018} and section~\ref{Low-rank FOM and GMRES}. 

In principle, the truncation steps can affect the convergence of the Krylov method and 
the well-established properties of Krylov schemes~(see, e.g., \cite{Saad2003}) may no longer hold. However, it has been numerically observed that Krylov methods with low-rank truncations often achieve the desired accuracy, 
even when the truncation strategy is particularly aggressive. See, e.g.,\cite{Stoll2015, Freitag2018}.

In this paper we establish some theoretical foundations to explain the converge of Krylov methods with low-rank truncations. In particular, the full orthogonalization method (FOM)~\cite[Section 6]{Saad2003} and the generalized minimal residual method (GMRES)~proposed in~\cite{Saad1986} are analyzed.

We assume that two different truncation steps are performed within our routine and, to show that the convergence is maintained, we interpret these truncations in two distinct ways.
First, the truncation performed after the matrix-vector product $\mathcal{A}v_m$ is seen as an inexact matrix-vector product and results coming from~\cite{Simoncini2003} are employed. Second, the low-rank truncations that take place during the orthogonalization procedure are viewed as a structured perturbation of the new basis vector that preserves orthogonality; the perturbed vector is still orthogonal with respect to the previous ones.

 We would like to underline the fact that the schemes studied in this paper significantly differ from \emph{tensorized} Krylov methods analysed in, e.g., \cite{Kressner2009/10}. Indeed, our $\mathcal{A}$ is not a \emph{Laplace-like} operator in general, i.e., $\mathcal{A}\neq\sum_{i=1}^pI_{n_1}\otimes\cdots\otimes I_{n_{i-1}}\otimes A_i\otimes I_{n_{i+1}}\otimes\cdots\otimes I_{n_p} $.

The following is a synopsis of the paper. In section~\ref{Low-rank FOM and GMRES} we review the low-rank formulation of FOM and GMRES and their convergence is proved in 
section~\ref{A converge result}. In particular, in section~\ref{Inexact matrix-vector products}-\ref{Structured perturbations of the basis} the two different interpretations of the 
low-rank truncation steps are presented. 
Some implementation aspects of these low-rank truncations are discussed in section~\ref{sec:ranktrunc}. It is well known that Krylov methods must be equipped with effective preconditioning 
techniques in order to achieve a fast convergence in terms of number of iterations. Due to some peculiar aspects of our setting, the preconditioners must be carefully designed as we discuss in
section~\ref{Preconditioning}. 
Short recurrence methods like CG, MINRES and BICGSTAB can be very appealing in our context due to their small memory requirements and low computational efforts per iteration.
Even though their analysis can be cumbersome since the computed basis is not always orthogonal (e.g., the orthogonality may be lost in finite precision arithmetic), their application to the solution of~\eqref{eq_main} is discussed in section~\ref{Short recurrence methods}.
Several numerical examples reported in section~\ref{Numerical examples} support our theoretical analysis. The paper finishes with 
some conclusions given in section~\ref{Conclusions}.

Throughout the paper we adopt the following notation.
 The matrix inner product is
defined as $\langle X, Y \rangle_F = \mbox{trace}(Y^T X)$ so that the induced norm is $\|X\|_F= \sqrt{\langle X, X\rangle_F}$. In the paper we continuously use the 
identity $\text{vec}(Y)^T\text{vec}(X)=\langle X,Y\rangle_F$ so that $\|\text{vec}(X)\|_2^2=\|X\|_F^2$. Moreover, the cyclic property of the trace operator allows for a cheap evaluation of 
matrix inner products with low-rank matrices. Indeed, if $M_i,N_i\in\mathbb{R}^{n\times r_i}$, $r_i\ll n$, $i=1,2$, 
$\langle M_1N_1^T,M_2N_2^T\rangle_F=\text{trace}(N_2M_2^TM_1N_1^T)=\text{trace}((M_2^TM_1)(N_1^TN_2))$ and only matrices of small dimensions $r_i$ are involved in such a computation. 
Therefore, even if it is not explicitly stated, we will always assume that matrix inner products with low-rank matrices are cheaply computed without assembling any dense $n\times n$ matrix. For sake of simplicity we will omit the subscript in $\|\cdot\|_F$ and write only $\|\cdot\|$.

The $k$-th singular value of a matrix $M\in\mathbb{R}^{m_1\times m_2}$ is denoted by $\sigma_k(M)$, where the singular values are assumed to be ordered in a decreasing fashion. %whereas, if $M$ is full rank,
The condition number of $M$ is denoted by $\kappa(M)=\sigma_1(M)/\sigma_p(M)$, $p=\mathrm{rank}(M)=\argmin_i\lbrace\sigma_i(M)\neq 0\rbrace$.
%$m=\min(m_1,m_2)$%, 
%denotes its condition number.

As already mentioned, $I_n$ denotes the identity matrix of order $n$ and the subscript
is omitted whenever the dimension of $I$ is clear from the context. The $i$-th canonical basis vector of $\mathbb{R}^n$ is denoted by $e_i$ while $\mathbf{0}_m$ is a vector of length $m$ whose entries are all zero. 

The brackets $[\cdot]$ are used to concatenate matrices of conforming dimensions. In particular, a Matlab-like notation is adopted and $[M,N]$ denotes the matrix obtained by stacking $M$ and $N$ one next to the other whereas $[M;N]$ the one obtained by stacking $M$ and $N$ one of top of each other, i.e., $[M;N]=[M^T,N^T]^T$.
The notation $\text{diag}(M,N)$ is used to denote the block diagonal matrix with diagonal blocks $M$ and $N$.

%%%%%%%%%%%%%%%%%%%%%%%%%%%%%%%%%%%%%%%%%%%%%%%%%%
\section{Low-rank FOM and GMRES}\label{Low-rank FOM and GMRES}
In this section we revise the low-rank formulation of FOM (LR-FOM) and GMRES (LR-GMRES) for the solution of the multiterm matrix equation~\eqref{eq_main}.

Low-rank Krylov methods compute an approximate solution $X_m\approx X$ of the form 
\begin{equation}\label{SolutionExpression}
\text{vec}(X_m)=x_0+V_my_m. 
\end{equation}
In the following we will always assume the initial guess $x_0$ to be the zero vector $\mathbf{0}_n$ and in Remark~\ref{Remark_initialguess} such a choice is motivated. 
Therefore, the $m$ orthonormal columns of $V_m=[v_1,\ldots,v_m]\in\mathbb{R}^{n^2\times m}$ in~\eqref{SolutionExpression} span the Krylov subspace~\eqref{eq.DefKrylovSpace} and $y_m\in\mathbb{R}^m$.

One of the peculiarities of low-rank Krylov methods is that the basis vectors must be stored in low-rank format. We thus write $v_j=\text{vec}(\mathcal{V}_{1,j}\mathcal{V}_{2,j}^T)$ where $\mathcal{V}_{1,j},~\mathcal{V}_{2,j}\in\mathbb{R}^{n\times s_j}$, $s_j\ll n$, for all $j=1,\ldots,m$.

The basis $V_m$ can be computed by a reformulation of the underlying Arnoldi process~(see, e.g., \cite[Section 6.4]{Saad2003}) that exploits the Kronecker structure of $\mathcal{A}$ and the low-rank format of the basis vectors.
In particular, at the $m$-th iteration,
the $n^2$-vector $\widehat v=\mathcal{A}v_m$ must be computed. For sparse matrices $A_i,~B_i$, a naive implementation of this operation costs $\mathcal{O}(\mathtt{nnz}(\mathcal{A}))$ floating point operations (flops) where $\mathtt{nnz}(\mathcal{A})$ denotes the number of nonzero entries of $\mathcal{A}$. However,  it can be replaced by the linear combination $\widehat V=\sum_{i=1}^p \left(A_i\mathcal{V}_{1,j}\right)\left(B_i\mathcal{V}_{2,j}\right)^T$, $\text{vec}(\widehat V)=\widehat v$, where $2ps_j$ matrix-vector products with matrices of order $n$ are performed. The cost of such operation is $\mathcal{O}((\max_i\mathtt{nnz}(A_i)+\max_i\mathtt{nnz}(B_i))ps_j)$ flops and it is thus much cheaper than computing $\widehat v$ naively via the matrix-vector product by $\mathcal{A}$ since $\mathtt{nnz}(\mathcal{A})=\mathcal{O}(\max_i\mathtt{nnz}(A_i)\cdot\max_i\mathtt{nnz}(B_i))$, $s_j$ is supposed to be small and $p$ is in general moderate. A similar argumentation carries over when (some of) the matrices $A_i,~B_i$ are not sparse but still allow efficient matrix vector products.
%\pnote{What about nonsparse $A_i,B_i$? will add quick remark}

 Moreover, since 
$$\widehat V=\sum_{i=1}^p \left(A_i\mathcal{V}_{1,j}\right)\left(B_i\mathcal{V}_{2,j}\right)^T=[A_1\mathcal{V}_{1,j},\ldots,A_p\mathcal{V}_{1,j}][B_1\mathcal{V}_{2,j},\ldots,B_p\mathcal{V}_{2,j}]^T=\widehat V_1 \widehat V_2^T,\quad \widehat V_1,\widehat V_2\in\mathbb{R}^{n\times ps_j},$$
the low-rank format is preserved in the computation of $\widehat V$.
In order to avoid an excessive increment in the column dimensions $ps_j$ of $\widehat V_1,\widehat V_2$, it is necessary to exercise a column compression of the factors $\widehat V_1$ and $\widehat V_2$, i.e., the matrices $(\widebar V_1, \widebar V_2)=\mathtt{trunc}(\widehat V_1,I,\widehat V_2,\varepsilon_{\mathcal{A}})$ are computed. With $\mathtt{trunc}(L,M,N,\varepsilon_{\text{trunc}})$ we denote any routine that computes low-rank approximations of the product $LMN^T$ with a desired accuracy of order $\varepsilon_{\mathtt{trunc}}$, so that, the matrices $\widebar V_1$, $\widebar V_2$ are such that $\|\widebar V_1 \widebar V_2^T- \widehat V_1 \widehat V_2^T\|/\|\widehat V_1 \widehat V_2^T\|=\varepsilon_{\mathcal{A}}$ with
$\widebar V_1$, $\widebar V_2\in\mathbb{R}^{n\times \widebar s}$, $\widebar s\leq ps_j$. Algorithm~\ref{trunc_routine1.1} illustrates a standard approach for such compressions that is based on thin QR-factorizations and a SVD thereafter; see, e.g.,~\cite[Section~2.2.1]{Kressner2011}, and used in the remainder of the paper. Some alternative truncation schemes are discussed in section~\ref{sec:ranktrunc}.

%%%%%%%%%%%%%%%%%%%%%%%%%%%%%%%%%%%%%%%%%%%%%%%%%%%%%%%%%%%%%%%%%%%%%%%%%%%%%%%%%%%%%%%
%TO FIX 
%%%%%%%%%%%%%%%%%%%%%%%%%%%%%%%%%%%%%%%%
%\setcounter{AlgoLine}{0}
%%%
\begin{algorithm}
%\algsetup{linenosize=\small}
%\SetLine %% new algorithm2e: \SetAlgoLined
\caption{$\mathtt{trunc}(L,M,N,\varepsilon_{\mathtt{trunc}})$}
\label{trunc_routine1.1}
\SetKwInOut{Input}{input}
\SetKwInOut{Output}{output}
%%%%%%%%%%% INPUT %%%%%%%%%%%
\Input{$L,N\in\mathbb{R}^{n\times r},$ $r\ll n$, $M\in\mathbb{R}^{r\times r}$, $\varepsilon_{\text{trunc}}>0$}
%%%%%%%%%%% OUTPUT %%%%%%%%%%%
\Output{$F,G\in\mathbb{R}^{n\times \widebar k}$, $\widebar k\leq r$, $\|FG^T-LMN^T\|/\|LMN^T\|=\varepsilon_{\mathtt{trunc}}$}
%%%%%%%%%%%%%%%%%%%%%%%%%%%%%%%%%%%
\BlankLine
\nl Compute skinny QR factorizations $Q_LR_L=L$, $Q_NR_N=N$\\
\nl Compute the SVD decomposition $U\Sigma W^T=R_LMR_N^T\in\mathbb{R}^{r\times r}$, $U=[u_1,\ldots,u_r]$, $W=[w_1,\ldots,w_r]$ $\Sigma=\text{diag}(\sigma_1,\ldots,\sigma_r)$, $\sigma_1\geq\cdots\geq\sigma_r\geq0$\\
\nl Find the smallest index $\widebar k$ such that $\sqrt{\sum_{i=\widebar k+1}^r\sigma_i^2}\leq\varepsilon_{\text{trunc}}\|\Sigma\|$\\
\nl Define $F:=Q_L([u_1,\ldots,u_{\widebar k}] \sqrt{\text{diag}(\sigma_1,\ldots,\sigma_{\widebar k})})$ and $G:=Q_N([w_1,\ldots,w_{\widebar k}] \sqrt{\text{diag}(\sigma_1,\ldots,\sigma_{\widebar k})})$
\end{algorithm}

%%%%%%%%%%%%%%%%%%%%%%%%%%%%%%%%%%%%%%%%%%%%%%%%%%%%%%%%%%%%%%%%%%%%%%%%%%%%%%%%%%%%%%%%

The vector $\text{vec}(\widebar V_1\widebar V_2^T)\approx \widehat v$ returned by the truncation algorithm is then orthogonalized with respect to the previous basis vectors $\text{vec}(\mathcal{V}_{1,j}\mathcal{V}_{1,j}^T)$, $j=1,\ldots,m$. 
Such an orthogonalization step can be implemented by performing, e.g., the modified Gram-Schmidt procedure and the low-rank format of the quantities involved can be exploited and maintained in the result. The vector formulation of the orthogonalization step is given by
\begin{equation}\label{GS_vector}
 \widetilde v= \text{vec}(\widebar V_1\widebar V_2^T)-\sum_{j=1}^m\left(\text{vec}(\mathcal V_{1,j}\mathcal V_{2,j}^T)^T\text{vec}(\widebar V_1\widebar V_2^T)\right)\text{vec}(\mathcal V_{1,j}\mathcal V_{2,j}^T),
\end{equation}
and, since $\text{vec}(\mathcal V_{1,j}\mathcal V_{2,j}^T)^T\text{vec}(\widebar V_1\widebar V_2^T)=\langle \mathcal V_{1,j}\mathcal V_{2,j}^T,\widebar V_1\widebar V_2^T\rangle_F$, we can reformulate~\eqref{GS_vector} as
$$\widetilde V=\widebar V_1\widebar V_2^T-\sum_{j=1}^mh_{j,m}\mathcal V_{1,j}\mathcal V_{2,j}^T=[\widebar V_1,\mathcal V_{1,1},\ldots,\mathcal V_{1,m}]\Theta_m[\widebar V_2,\mathcal V_{2,1},\ldots,\mathcal V_{2,m}]^T,\quad h_{j,m}=\langle \mathcal V_{1,j}\mathcal V_{2,j}^T,\widebar V_1\widebar V_2^T\rangle_F,$$
where $\Theta_m=\text{diag}(I_{\widebar s},-h_{1,m}I_{s_1},\ldots,-h_{m,m}I_{ s_m})$, $\text{vec}(\widetilde V)=\widetilde v$, and the $m$ coefficients $h_{j,m}$ are collected in the $m$-th column of an upper Hessenberg matrix $H_m\in\mathbb{R}^{m\times m}$. 
Obviously, the result $\widetilde V$ has factors with increased column dimensions such that a truncation of
 the matrix 
$[\widebar V_1,\mathcal V_{1,1},\ldots,\mathcal V_{1,m}]\Theta_m[\widebar V_2,\mathcal V_{2,1},\ldots,\mathcal V_{2,m}]^T$ becomes necessary. In particular, if $\varepsilon_{\mathtt{orth}}$ is a given threshold, we compute 
\begin{equation}\label{truncation_orth}
   (\widetilde V_{1},\widetilde V_{2})=\mathtt{trunc}([\widebar V_1,\mathcal V_{1,1},\ldots,\mathcal V_{1,m}],\Theta_m,[\widebar V_2,\mathcal V_{2,1},\ldots,\mathcal V_{2,m}],\varepsilon_{\mathtt{orth}}).
\end{equation}
The result in~\eqref{truncation_orth} is then normalized to obtained the $(m+1)$-th basis vector, namely $\mathcal{V}_{1,m+1}=\widetilde V_{1}/\sqrt{\|\widetilde V_{1}\widetilde V_{2}^T\|}$ and 
$\mathcal{V}_{2,m+1}=\widetilde V_{2}/\sqrt{\|\widetilde V_{1}\widetilde V_{2}^{T}\|}$.
%%
%
% What is done in the actual computation is to perform $m$ low-rank truncations according to the following scheme
% %%
% {\small
% \begin{equation}
% \begin{array}{l}
% (\widetilde V_{1,1}^{(m)},\widetilde V_{2,1}^{(m)})=\mathtt{trunc}([\widebar V_1, \mathcal{V}_{1,1}],\text{diag}(I_{\widebar s_1},-h_{1,m}I_{s_1}),[\widebar V_2 ,\mathcal{V}_{2,1}],\varepsilon_{\mathtt{orth}}), \quad h_{1,m}=\langle \mathcal{V}_{1,1}\mathcal{V}_{2,1}^T,\widebar V_1\widebar V_2^T\rangle_F,\\
%  \\ 
%   (\widetilde V_{1,2}^{(m)},\widetilde V_{2,2}^{(m)})=\mathtt{trunc}([\widetilde V_{1,1}^{(m)}, \mathcal{V}_{1,2}],\text{diag}(I_{\widetilde s_1},-h_{2,m}I_{s_2}),[\widetilde V_{2,1}^{(m)} ,\mathcal{V}_{2,2}],\varepsilon_{\mathtt{orth}}), \quad
% {\widetilde s_1}=\text{rank}(\widetilde V_{1,1})\\
%   \\
%  \vdots\\
%   \\
%    (\widetilde V_{1,m}^{(m)},\widetilde V_{2,m}^{(m)})=\mathtt{trunc}([\widetilde V_{1,m-1}^{(m)}, \mathcal{V}_{1,m}],\text{diag}(I_{\widetilde s_{m-1}},-h_{m,m}I_{s_m}),[\widetilde V_{2,m-1}^{(m)} ,\mathcal{V}_{2,m}],\varepsilon_{\mathtt{orth}}),\quad 
%   \widetilde s_{m-1}=\text{rank}(\widetilde V_{1,m-1}),
% \end{array}
% \end{equation}
% }
% %%
% where $h_{j,m}=\langle \mathcal{V}_{1,j}\mathcal{V}_{2,j}^T,\widetilde V_{1,j-1}\widetilde V_{2,j-1}^T\rangle_F$ for $j=2,\ldots,m$ and $\varepsilon_{\mathtt{orth}}$ is a given threshold.
%
The upper Hessenberg matrix $\underline{H}_m\in\mathbb{R}^{(m+1)\times m}$ is defined such that its square principal submatrix is given by $H_m$ and $e_{m+1}^T \underline{H}_me_m=h_{m+1,m}:=\|\widetilde V_{1}\widetilde V_{2}^T\|$.

The difference between FOM and GMRES lies in the computation of the vector $y_m$ in~\eqref{SolutionExpression}. In FOM a Galerkin condition on the residual vector 
\begin{equation}\label{Galerkin}
\text{vec}(C_1C_2^T)+\mathcal{A}\text{vec}(X_m) \perp \mathbf{K}_m(\mathcal{A},\text{vec}(C_1C_2^T)). 
\end{equation}
is imposed.
If no truncation steps are performed during the Arnoldi procedure, the Arnoldi relation 
\begin{equation}\label{exact_Arnoldi}
 \mathcal{A}V_m=H_mV_m+h_{m+1,m}\text{vec}(\mathcal{V}_{1,m+1}\mathcal{V}_{2,m+1}^T)e_m^T,
\end{equation}
is fulfilled and it is easy to show that imposing the Galerkin condition~\eqref{Galerkin} is equivalent to solving the $m\times m$ linear system 
\begin{equation}\label{FOM_linearsystem}
H_my_m^{fom}=\beta e_1, \quad\beta=\|C_1C_2^T\|,
\end{equation}
for $y_m=y_m^{fom}$. Moreover, in the exact setting where~\eqref{exact_Arnoldi} holds, the norm of the residual vector $\text{vec}(C_1C_2^T)+\mathcal{A}\text{vec}(X_m)$ can be cheaply computed as
$$\|\text{vec}(C_1C_2^T)+\mathcal{A}\text{vec}(X_m)\|=h_{m+1,m}|e_m^Ty^{fom}_m|.$$
See, e.g.,~\cite[Proposition 6.7]{Saad2003}. We show later that this is possible also when the low-rank truncations are performed and an \emph{inexact} version of~\eqref{exact_Arnoldi} is taken into account.

In GMRES, the vector $y_m=y_m^{gm}$ is computed by solving a least squares problem 
\begin{equation*}
y_m^{gm}=\argmin_{y_m}\|\text{vec}(C_1C_2^T)+\mathcal{A}V_my_m\|,
\end{equation*}
which corresponds to the Petrov-Galerkin orthogonality condition
\begin{equation}\label{MresGalerkin}
\text{vec}(C_1C_2^T)+\mathcal{A}\text{vec}(X_m) \perp \mathcal{A}\cdot\mathbf{K}_m(\mathcal{A},\text{vec}(C_1C_2^T)). 
\end{equation}
If~\eqref{exact_Arnoldi} holds, $y_m^{gm}$ can be computed as 
\begin{equation}\label{least_squares_GMRES}
y_m^{gm}=\argmin_{y_m}\|\beta e_1+\underline{H}_my_m\|,
\end{equation}
 and, following the discussion in~\cite[Section 6.5.3]{Saad2003}, this reduced least squares problem can be cheaply solved by applying $m$ Givens rotations $\Omega_i$. If $\underline{U}_m=\prod_{i=1}^m\Omega_i\underline{H}_m\in\mathbb{R}^{(m+1)\times m}$ is upper triangular and $\underline{g}_m=\beta\prod_{i=1}^m\Omega_ie_1\in\mathbb{R}^{m+1}$, then the vector $y_m^{gm}$  is given by the solution of the $m\times m$ linear system $U_my_m^{gm}=g_m$ where $U_m$ denotes the square principal submatrix of $\underline{U}_m$
and $g_m$ collects the first $m$ components of $\underline g_m$. Moreover,
$$\|\text{vec}(C_1C_2^T)+\mathcal{A}\text{vec}(X_m)\|=|e_{m+1}^T\underline{g}_m|.$$
See, e.g.,~\cite[Proposition 6.9]{Saad2003}. As for FOM, we will show that this is possible also in the case of GMRES
equipped with low-rank truncations.

If at the $m$-th iteration the residual norm $\|\text{vec}(C_1C_2^T)+\mathcal{A}V_my_m\|$ is sufficiently small\footnote{$y_m=y_m^{fom}$ or $y_m=y_m^{gm}$.}, we recover the solution $X_m$. 
Clearly, the full $X_m$ is not constructed explicitly as this is a large, dense matrix. However, since we have assumed that the solution $X$ to~\eqref{eq_main} admits accurate low-rank approximations, 
we can compute low-rank factors $S_1$, $S_2\in\mathbb{R}^{n\times t}$, $t\ll n$, such that $S_1S_2^T\approx X$. Also this operation can be performed by exploiting the low-rank format of the basis vectors. In particular, if 
$\Upsilon=\text{diag}((e_1^Ty_m)I_{ s_1},\ldots,(e_m^Ty_m)I_{ s_m})$, then
\begin{equation}\label{recover_Xm}
 (S_{1}, S_{2})=\mathtt{trunc}([\mathcal{V}_{1,1}, \ldots \mathcal{V}_{1,m}],\Upsilon,[\mathcal{V}_{1,2}, \ldots \mathcal{V}_{2,m}],\varepsilon).
 \end{equation}
% 
% %%
% \begin{equation}\label{recover_Xm}
% \begin{array}{l}
%  (S_{1,1}, S_{2,1})=\mathtt{trunc}( [\mathcal{V}_{1,1}, \mathcal{V}_{1,2}],\text{diag}((e_1^Ty_m)I_{s_1},(e_2^Ty_m)I_{s_2}),[\mathcal{V}_{2,1}, \mathcal{V}_{2,1}],\nu),\\
%  \\
%  (S_{1,2}, S_{2,2})= \mathtt{trunc}( [S_{1,1}, \mathcal{V}_{1,3}],\text{diag}(I_{t_1},(e_3^Ty_m)I_{s_3}),[S_{2,1}, \mathcal{V}_{2,3}],\nu), \quad t_1:=\text{rank}(S_{1,1}),\\
%  \\
%  \vdots\\
%  \\
%  (S_{1}, S_{2})=\mathtt{trunc}([S_{1,m-2}, \mathcal{V}_{1,m}],\text{diag}(I_{t_{m-2}},(e_m^Ty_m)I_{s_m}),[S_{2,m-2}, \mathcal{V}_{2,m}],\nu), \quad t_{m-2}:=\text{rank}(S_{1,m-2}).\\
% \end{array}
% \end{equation}
% %%

The low-rank FOM and GMRES procedures are summarized in
Algorithm~\ref{LR_FOM}. For sake of simplicity, we decide to collect the two routines in the same pseudo-algorithm as they differ only in the convergence check if a Givens rotations approach similar to the one presented for GMRES is adopted also for FOM. This allows for a cheap evaluation of the residual norm without solving the linear system~\eqref{FOM_linearsystem} at each iteration. 

%Algorithm~\ref{LR_GMRES} sketches the low-rank GMRES method.

%%%%%%%%%%%%%%%%%%%%%%%%%%%%%%%%%%%%%%%%%%%%%%%%%%%%%%%%%%%%%%%%%%%%%%%%%%%%%%%%%%%%%%%
%TO FIX 
%%%%%%%%%%%%%%%%%%%%%%%%%%%%%%%%%%%%%%%%
%%%
\setcounter{AlgoLine}{0}
%%%
\begin{algorithm}
%\algsetup{linenosize=\small}
%\SetLine %% new algorithm2e: \SetAlgoLined
\caption{LR-FOM and LR-GMRES\label{LR_FOM}}
\SetKwInOut{Input}{input}\SetKwInOut{Output}{output}
%%%%%%%%%%% INPUT %%%%%%%%%%%
\Input{$A_i,B_i\in\mathbb{R}^{n\times n},$ for $i=1,\ldots,p$, $C_1,C_2\in\mathbb{R}^{q\times n}$, $m_{\max}$, $\varepsilon_{\mathcal{A}}$, $\varepsilon_{\mathtt{orth}}$, $\varepsilon>0$}
%%%%%%%%%%% OUTPUT %%%%%%%%%%%
\Output{$S_1,S_2\in\mathbb{R}^{n\times t}$, $t\ll n$, $S_1S_2^T\approx X$ approximate solution to~\eqref{eq_main}}
%%%%%%%%%%%%%%%%%%%%%%%%%%%%%%%%%%%
\BlankLine
\nl Compute $\beta=\|C_1C_2^T\|$ and set $\Omega_1=1$, $\underline{g}_1=\beta e_1$, $\mathcal{V}_{1,1}= C_1/\sqrt{\beta}$ and
  $\mathcal{V}_{2,1}= C_2/\sqrt{\beta}$\\   \For{$m=1, 2,\dots,$ till $m_{\max}$}{
  \nl Set $\widehat V_1=[A_1\mathcal{V}_{1,m},\ldots, A_p\mathcal{V}_{1,m}]$ and $\widehat V_2=[B_1\mathcal{V}_{2,m},\ldots, B_p\mathcal{V}_{2,m}]$\label{alg_line_product}\\
  \nl Compute $(\widebar V_1, \widebar V_2)=\mathtt{trunc}(\widehat V_1,I,\widehat V_2,\varepsilon_{\mathcal{A}})$ \label{algo_firsttrunc}\\
  \nl Set $h_{j,m}=0$ for $j=1,\ldots,m$ \\
  \For{$\ell=1,2$}{
  \For{$j=1,\ldots,m$}{
  \nl Compute $h_{j,m}=h_{j,m}+\langle \mathcal{V}_{1,j}\mathcal{V}_{2,j}^T,\widebar V_1\widebar V_2^T\rangle_F$ and collect it in $\underline{H}_me_m\in\RR^{m+1}$  
  }
  \nl Set $\Theta_m=\text{diag}(I_{\widebar s},-h_{1,m}I_{s_1},\ldots,-h_{m,m}I_{ s_m})$, $\widebar s=\text{rank}(\widebar V_1)$\\
  \nl Compute $(\widebar V_{1},\widebar V_{2})=\mathtt{trunc}([\widebar V_1,\mathcal V_{1,1},\ldots,\mathcal V_{1,m}],\Theta_m,[\widebar V_2,\mathcal V_{2,1},\ldots,\mathcal V_{2,m}],\varepsilon_{\mathtt{orth}})$
  \label{algo_secondtrunc}}
  \nl Set $e_{m+1}^T\underline{H}_me_m=h_{m+1,m}=\|\widebar V_1\widebar V_2^T\|$\\
  \nl Set $\mathcal{V}_{1,m+1}=\widebar V_1/\sqrt{h_{m+1,m}}$ and
  $\mathcal{V}_{2,m+1}=\widebar V_2/\sqrt{h_{m+1,m}}$\\  
  \eIf{$m=1$}{
  \nl Set  $\underline{U}_1=\text{diag}(\Omega_1,1) \underline{H}_1e_1$ 
  }{
   \nl Set $\underline{U}_m=[[\underline{U}_{m-1};\mathbf{0}_m^T],\prod_{i=1}^m\text{diag}(\Omega_i,I_{m+1-i}) \underline{H}_me_m]$
   }
  \If{{\rm \textbf{FOM} and} $|h_{m+1,m}(e_m^T\underline{g}_m)/(e_m^T\underline{U}_me_m)|<\varepsilon\cdot\beta$}{ 
\nl \textbf{Break} and go to \textbf{\ref{algo_line_afterloop}} }
\nl Compute $\Omega_{m+1}\in\mathbb{R}^{(m+1)\times(m+1)}$ such that $\underline{U}_m=\Omega_{m+1}\underline{U}_m$ is upper triangular \\%\pnote{Something is not correct here. The way $\underline{U}_m$ is written above, it seems already upper triangular}%\\
\nl Set $\underline{g}_{m+1}=\text{diag}(\Omega_{m+1},1)[\underline{g}_m;0]$
\label{gmres_rescomp}\\
 \If{{\rm \textbf{GMRES} and} $|e_{m+1}^T\underline{g}_{m+1}|<\varepsilon\cdot\beta$}{ 
\nl \textbf{Break} and go to \textbf{\ref{algo_line_afterloop}} }

  }
  
  \nl Set $U_m=[I_m,\mathbf{0}_m]\underline{U}_m[I_m;\mathbf{0}_m^T]\in\mathbb{R}^{m\times m}$ and $g_m=[I_m,\mathbf{0}_m]\underline{g}_m\in\mathbb{R}^m$\label{algo_line_afterloop}\\
  \nl Compute $y_m=U_m^{-1}g_m$ \\
  \nl Set $\Upsilon=\text{diag}((e_1^Ty_m)I_{ s_1},\ldots,(e_m^Ty_m)I_{ s_m})$
  \\
  \nl Compute $(S_{1}, S_{2})=\mathtt{trunc}([\mathcal{V}_{1,1}, \ldots \mathcal{V}_{1,m}],\Upsilon,[\mathcal{V}_{1,2}, \ldots \mathcal{V}_{2,m}],\varepsilon)$ \label{algo_lasttrunc}
  % Set $S_1=(e_1^Ty^{fom}_m)\mathcal{V}_{1,1}$ and $S_2=\mathcal{V}_{2,1}$\\
 % \For{$i=2,\ldots,m$}{
 % \nl Compute $(S_1,S_2)=\mathtt{trunc}([S_1, \mathcal{V}_{1,i}],\text{diag}(I_{\text{rank}(S_1)},(e_i^Ty^{fom}_m)I_{s_i}),[S_2 ,\mathcal{V}_{2,i}],\varepsilon)$
  %}
\end{algorithm}
%%
%%%%%%%%%%%%%%%%%%%%%%%%%%%%%%%%%%%%%%%%%%%%%%%%%%%%%%%%%%%%%%%%%%%%%%%%%%%%%%%%%%%%%%%

At each iteration step $m$ of Algorithm~\ref{LR_FOM} we perform three low-rank truncations\footnote{One after the application of $\mathcal{A}$ in line~\ref{algo_firsttrunc}, and two during the orthogonalization procedure in line~\ref{algo_secondtrunc}, at the end of each of the two loops of the modified Gram-Schmidt method.} and these operations  %remarkably 
%\pnote{Wait.. in each step only two truncations, right? The third one is only used in the end, right?}
substantially influence the overall solution procedure. If the truncation tolerances $\varepsilon_{\mathcal{A}}$ and $\varepsilon_{\mathtt{orth}}$ are chosen too large, the whole Krylov method my break down. Therefore, in the following sections we discuss how to adaptively choose the truncation tolerances $\varepsilon_{\mathcal{A}}$ and $\varepsilon_{\mathtt{orth}}$ to maintain convergence.
Moreover, the low-rank truncation does have its own computational workload which can be remarkable, especially if the ranks of the basis vectors involved is quite large. 
%\textbf{Here one would argue that a short-recurrence method does not have this high effort.}
In section~\ref{sec:ranktrunc} we discuss some computational appealing alternatives to Algorithm~\ref{trunc_routine1.1}.

%%%%%%%%%%%%%%%%%%%%%%%%%%%%%%%%%%%%%%%%%%%%%%%%%%
\section{A convergence result}\label{A converge result}
In this section we show that the convergence of LR-FOM and LR-GMRES is guaranteed if the thresholds $\varepsilon_{\mathcal{A}}$ and $\varepsilon_{\mathtt{orth}}$ for the low-rank truncations in line~\ref{algo_firsttrunc} 
and~\ref{algo_secondtrunc} of Algorithm~\ref{LR_FOM} are properly chosen and if the routine used in the truncation steps satisfies certain properties.

The truncation that takes place in line~\ref{algo_lasttrunc}, after the iterative process terminated, to recover the low-rank factors of the approximate solution is not discussed. Indeed, this does not affect the convergence of the Krylov method and it is justified by assuming that the 
exact solution $X$ admits low-rank approximations.

%%%%%%%%%%%%%%%%%%%%%%%%%%%%%%%%%%%%%%%%%%%%%%%%%%%
\subsection{Inexact matrix-vector products}\label{Inexact matrix-vector products}
We start by analyzing the truncation step in line~\ref{algo_firsttrunc} of Algorithm~\ref{LR_FOM} assuming, for the moment, that the one in line~\ref{algo_secondtrunc} is not performed. In this way the generated basis $V_m$ is ensured to be orthogonal. 
In section~\ref{Structured perturbations of the basis} we will show that the truncation in line~\ref{algo_secondtrunc} of Algorithm~\ref{LR_FOM} preserves the orthogonality of the constructed basis so that the
results we show here still hold.

The low-rank truncation performed in line~\ref{algo_firsttrunc} of Algorithm~\ref{LR_FOM} can be understood as an inexact matrix-vector product with $\mathcal{A}$. Indeed, at the $m$-th iteration,
we can write
$$\widehat V_1\widehat V_2^T=\widebar V_1\widebar V_2^T+E_m,$$
where $E_m$ is the matrix discarded when $\mathtt{trunc}(\widehat V_1,I,\widehat V_2,\varepsilon_{\mathcal{A}})$ is applied so that $\|E_m\|/\|\widehat V_1\widehat V_2^T\|\leq\varepsilon_{\mathcal{A}}$.
Therefore, we have
$$\text{vec}(\widebar V_1\widebar V_2^T)=\mathcal{A}\text{vec}(\mathcal{V}_{1,m}\mathcal{V}_{2,m}^T)-\text{vec}(E_m),\quad \|\text{vec}(E_m)\|\leq \varepsilon_{\mathcal{A}}\cdot 
\|\mathcal{A}\text{vec}(\mathcal{V}_{1,m}\mathcal{V}_{2,m}^T)\|,$$
and the vector $\text{vec}(\widebar V_1\widebar V_2^T)$ can thus be seen as the result of an inexact matrix-vector product by $\mathcal{A}$.

Following the discussion in~\cite{Simoncini2003}, the Arnoldi relation~\eqref{exact_Arnoldi} must be replaced with the inexact counterpart
\begin{equation}\label{inexact_Arnoldi}
 \mathcal{A}V_m-[\text{vec}(E_1),\ldots,\text{vec}(E_m)]=V_{m}H_m+h_{m+1,m}\text{vec}(\mathcal{V}_{1,m+1}\mathcal{V}_{2,m+1}^T)e_m^T,
\end{equation}
and $\text{Range}(V_m)$ is no longer a Krylov subspace generated by $\mathcal{A}$. 

The vectors $y_m^{fom}$ and $y_m^{gm}$ can be still calculated as in~\eqref{FOM_linearsystem} 
and~\eqref{least_squares_GMRES}, respectively, but these are no longer equivalent to imposing the Galerkin and Petrov-Galerkin conditions~\eqref{Galerkin}-\eqref{MresGalerkin} since the Arnoldi 
relation~\eqref{exact_Arnoldi} 
no longer holds; different constraints must be taken into account. 
\begin{Prop}[See~\cite{Simoncini2003}]\label{Prop_constraints}
Let~\eqref{inexact_Arnoldi} hold and define $W_m=\mathcal{A}V_m-[\text{vec}(E_1),\ldots,\text{vec}(E_m)]$. If $y_m^{gm}$ is computed as in~\eqref{least_squares_GMRES}, then $q_m^{gm}:=W_my_m^{gm}$ is such that
$$ q_m=\argmin_{q\in\text{Range}(W_m)}\|\text{vec}(C_1C_2^T)+q\|.$$
Similarly, if $y_m^{fom}$ is computed as in~\eqref{FOM_linearsystem}, then $q_m^{fom}:=W_my_m^{fom}$ is such that
$$\text{vec}(C_1C_2^T)+q_m\perp \text{Range}(V_m).$$
\end{Prop}

Consequently, $H_m$ is not a true Galerkin projection of $\mathcal{A}$ onto Range$(V_m)$.
One may want to compute the vectors $y_m^{fom}$ and $y_m^{gm}$ by employing the true projection $T_m:=V_m^T\mathcal{A}V_m=H_m+V_m^T[\text{vec}(E_1),\ldots,\text{vec}(E_m)]$ in place of $H_m$ in ~\eqref{FOM_linearsystem}-\eqref{least_squares_GMRES} so that the reduced problems represent a better approximation (cf.~\cite{Gue10}) of the original equation and the orthogonality conditions imposed are in terms of the true residual.
However, the computation of $T_m$ requires to store the matrix $[\text{vec}(E_1),\ldots,\text{vec}(E_m)]$ and this is impracticable as the benefits in terms of memory demand coming from the low-rank truncations are completely lost due to the allocation of both $V_m$ and $[\text{vec}(E_1),\ldots,\text{vec}(E_m)]$. 
A different option is to store the matrix $\mathcal{A}V_m$ and compute an explicit projection of $\mathcal{A}$ onto the current subspace, but also this strategy leads to an unfeasible increment in the memory requirements of the overall solution process as the storage demand grows of a factor $p$. Therefore, in all the numerical experiments reported in section~\ref{Numerical examples},
the matrix $H_m$ arising from the orthonormalization procedure is employed in the computation of $y_m^{fom}$ and  $y_m^{gm}$.

If~\eqref{inexact_Arnoldi} holds and $\text{vec}(X_m)=V_my_m$ is the approximate solution to~\eqref{eq:linear_system} computed by projection onto $\text{Range}(V_m)$, then, at the $m$-th iteration,
the true residual vector can be expressed as
\begin{equation}
 r_m=\text{vec}(C_1C_2^T)+\mathcal{A}\text{vec}(X_m)=\text{vec}(C_1C_2^T)+\mathcal{A}V_my_m=\widetilde r_m-[\text{vec}(E_1),\ldots,\text{vec}(E_m)]y_m,
\end{equation}
where $\widetilde r_m$ is the computed residual vector.

In~\cite[Section 4]{Simoncini2003} it has been shown that the residual gap $\delta_m:=\|r_m-\widetilde r_m\|$ between the true residual and the computed one can be bounded by
$$\delta_m\leq \sum_{j=1}^m \|E_j\|\cdot|e_j^Ty_m|.$$
Since $|e_j^Ty_m|$ decreases as the the iterations proceed~(see, e.g., \cite[Lemma 5.1-5.2]{Simoncini2003}), $\| E_m\|$ is allowed to increase while still maintaining a small residual gap and preserving the convergence of the overall 
solution process. This phenomenon is often referred to as \textit{relaxation}.
\begin{theorem}[See~\cite{Simoncini2003}]\label{Theorem_inexactProd}
 Let $\varepsilon>0$ and let $r_m^{gm}:=\text{vec}(C_1C_2^T)+\mathcal{A}V_my_m^{gm}$ be the true 
%\pnote{true GMRES residuals?}
GMRES residual after $m$ iterations of the inexact Arnoldi procedure. If for every $k\leq m$,
 \begin{equation}\label{GMRES_bound_inexact1}
  \|E_k\|\leq \frac{\sigma_m(\underline{H}_m)}{m}\frac{1}{\|\widetilde r_{k-1}^{gm}\|}\varepsilon,
 \end{equation}
then $\|r_m^{gm}-\widetilde r_m^{gm}\|\leq\varepsilon$. Moreover, if 
\begin{equation}\label{GMRES_bound_inexact2}
  \|E_k\|\leq \frac{1}{m\kappa(\underline{H}_m)}\frac{1}{\|\widetilde r_{k-1}^{gm}\|}\varepsilon,
 \end{equation}
then $\|(V_{m+1}\underline{H}_m)^Tr_m^{gm}\|\leq\varepsilon$.

Similarly, if $r_m^{fom}:=\text{vec}(C_1C_2^T)+\mathcal{A}V_my_m^{fom}$ is the true FOM residual after $m$ iterations of the inexact Arnoldi procedure, and if for every $k\leq m$,
 \begin{equation}\label{FOM_bound_inexact}
  \|E_k\|\leq \frac{\sigma_m(H_m)}{m}\frac{1}{\|\widetilde r_{k-1}^{gm}\|}\varepsilon,
 \end{equation}
then $\|r_m^{fom}-\widetilde r_m^{fom}\|\leq\varepsilon$ and $\|V_{m}^Tr_m^{fom}\|\leq\varepsilon$.
\end{theorem}
Notice that the bound in~\eqref{FOM_bound_inexact} depends on the norm of the computed GMRES residual. This can be easily computed when Algorithm~\ref{LR_FOM} is performed as 
$\|\widetilde r_m^{gm}\|=|e_{m+1}^T\underline{g}_{m+1}|$ in line~\ref{gmres_rescomp} of Algorithm~\ref{LR_FOM}. However, if the FOM residual $\widetilde r_{k-1}^{fom}$ exists for every $k\leq m$, 
$\|\widetilde r_{k-1}^{gm}\|$ can be replaced by $\|\widetilde r_{k-1}^{fom}\|$ in~\eqref{FOM_bound_inexact}.
%Notice that if the computed residual norm decrease, the right hand sides in~\eqref{GMRES_bound_inexact1}-\eqref{GMRES_bound_inexact2}-\eqref{FOM_bound_inexact} are allowed to increase to some extent, a situation often referred to as \textit{relaxation}.
%Hence, as the low-rank Krylov methods progresses, the truncation can be done more and more aggressive.

The quantities involved in the estimates~\eqref{GMRES_bound_inexact1}-\eqref{GMRES_bound_inexact2}-\eqref{FOM_bound_inexact} are not available at iteration $k<m$ making the latter of theoretical interest only. To have practically usable truncation thresholds, the quantities in \eqref{GMRES_bound_inexact1}-\eqref{GMRES_bound_inexact2}-\eqref{FOM_bound_inexact} must be approximated with computable values.  Following the suggestions in~\cite{Simoncini2003}, we can replace $m$ by the maximum number $m_{\max}$ of allowed iterations, $\sigma_{m_{\max}}(\underline{H}_{m_{\max}})$
is replaced by $\sigma_{n^2}(\mathcal{A})$, and we approximate $\sigma_1(\underline{H}_{m_{\max}})$ by $\sigma_1(\mathcal{A})$ when computing $\kappa(\underline{H}_{m_{\max}})$ in~\eqref{GMRES_bound_inexact2}.
The extreme singular values of $\mathcal{A}$ can be computed once and for all at the beginning of the iterative procedure, e.g., 
by the Lanczos method that must be carefully designed to avoid the construction of $\mathcal{A}$ and to exploit its Kronecker structure. Approximations of $\sigma_{1}(\mathcal{A})$ and
$\sigma_{n^2}(\mathcal{A})$ coming, e.g., from some particular features of the problem of interest, can be also employed.
To conclude, we propose to use the following practical truncation thresholds $\varepsilon_{\mathcal{A}}^{(k)}$ in line~\ref{algo_firsttrunc} of Algorithm~\ref{LR_FOM} in place of $\varepsilon_{\mathcal{A}}$:
%\begin{subequation}\label{GMRES_bound_inexact_prac}
\begin{align}\label{GMRES_bound_inexact1_prac}
  \|E_k\|&\leq\varepsilon_{\mathcal{A}}^{(k)}=
	\begin{cases}
	\frac{c_1}{m_{\max}}\frac{1}{\|\widetilde r_{k-1}^{gm}\|}\varepsilon,\quad c_1\approx \sigma_{n^2}(\mathcal{A}),\\%\label{GMRES_bound_inexact2_prac}
	\frac{1}{m_{\max}c_2}\frac{1}{\|\widetilde r_{k-1}^{gm}\|}\varepsilon,\quad c_2\approx \kappa(\mathcal{A}),%
\end{cases} 
\end{align}
for LR-GMRES, and
 \begin{equation}\label{FOM_bound_inexact_prac}
  \|E_k\|\leq\varepsilon_{\mathcal{A}}^{(k)}=\frac{c_1}{m_{\max}}\frac{1}{\|\widetilde r_{k-1}^{gm}\|}\varepsilon,
 \end{equation}
 for LR-FOM.
Allowing $\|E_k\|$ to grow is remarkably important in our setting, especially for the memory requirements of the overall procedure.
Indeed, if the truncation step in line~\ref{algo_firsttrunc} of Algorithm~\ref{LR_FOM} is not performed, the rank of the 
basis vectors increases very quickly as, at the $m$-th iteration, we have
$$\text{rank}(\mathcal{V}_{1,m}\mathcal{V}_{2,m}^T)\leq qp^m.$$
Therefore, at the first iterations the rank of the basis vectors is low by construction and having a very stringent tolerance in the computation of their low-rank approximations is not an issue.
When the iterations proceed, the rank of the basis vectors increases but, at the same time, the increment in the thresholds for computing low-rank approximations of such vectors leads to more 
aggressive truncations with consequent remarkable gains in the memory allocation.

The interpretation of the truncation in line~\ref{algo_firsttrunc} of Algorithm~\ref{LR_FOM} in terms of an inexact Krylov procedure has been already proposed in~\cite{Dolgov2013} for the more general case of GMRES applied to~\eqref{eq:linear_system} where $\mathcal{A}$ is a tensor and the approximate solution is represented in the tensor-train (TT) format.
However, also in the tensor setting, the results in Theorem~\ref{Theorem_inexactProd} hold if and only if the matrix $V_m$ has orthonormal columns. In general, the low-rank truncation in line~\ref{algo_secondtrunc} can destroy the orthogonality of basis. In the next section we show that $V_m$ has orthogonal columns if the truncation step is performed in an appropriate way. 
%by Algorithm~\ref{trunc_routine1.1}--\ref{trunc_routine2}.

We first conclude this section with a couple of remarks.
\begin{remark}\label{Remark_initialguess}
  We have always assumed the initial guess $x_0\in\mathbb{R}^n$ in~\eqref{SolutionExpression} to be zero. This choice is motivated by the discussion in~\cite[Section 3]{Simoncini2003}, \cite{LieStr12} where 
the authors show how this is a good habit in the framework of inexact Krylov methods. 
\end{remark}

\begin{remark}\label{Remark_residualnorm}
 Since
 $$\|r_m\|\leq \|\widetilde r_m\|+\sum_{j=1}^m \|E_j\|\cdot|e_j^Ty_m|\leq \|\widetilde r_m\|+\sum_{j=1}^m \varepsilon_{\mathcal{A}}^{(j)}\cdot|e_j^Ty_m|,$$
%%%
where $\varepsilon_{\mathcal{A}}^{(j)}$ denotes one of the values in~\eqref{GMRES_bound_inexact1_prac}-\eqref{FOM_bound_inexact_prac}
depending on the selected procedure,
the quantity $\|\widetilde r_m\|+\sum_{j=1}^m \varepsilon_{\mathcal{A}}^{(j)}\cdot|e_j^Ty_m|$ must be computed to have a reliable stopping criterion in Algorithm~{\ref{LR_FOM}}. This means that the 
linear system $U_my_m=g_m$ has to be solved at each iteration $m$. This does not significantly increase the computational workload because $U_m\in\mathbb{R}^{m\times m}$ is of small dimension and already given in triangular form.
\end{remark}

%%%%%%%%%%%%%%%%%%%%%%%%%%%%%%%%%%%%%%%%%%%%%%%%%%%
\subsection{Structured perturbations of the basis}\label{Structured perturbations of the basis}
In this section we show how the low-rank truncations performed during the Gram-Schmidt procedure in line~\ref{algo_secondtrunc} of Algorithm~\ref{LR_FOM} preserve the orthogonality of the basis, i.e.,
$V_m$ is still an orthonormal matrix, and the results presented in section~\ref{Inexact matrix-vector products} are still valid.
\begin{Prop}\label{Prop_orthBasis}
 The matrix  $V_{m+1}=[\text{vec}(\mathcal{V}_{1,1}\mathcal{V}_{2,1}^T),\ldots,\text{vec}(\mathcal{V}_{1,m+1}\mathcal{V}_{2,m+1}^T)]\in\mathbb{R}^{n^2\times (m+1)}$
 computed by performing $m$ iterations of Algorithm~\ref{LR_FOM} where the low-rank truncations are computed by Algorithm~\ref{trunc_routine1.1} has orthonormal columns.
\end{Prop}
\begin{proof}
 At the $m$-th iteration, the $(m+1)$-th basis vector is computed by performing~\eqref{truncation_orth} and then normalizing the result.
 In particular, if $\Theta_m=\text{diag}(I_{\widebar s},-h_{1,m}I_{s_1},\ldots,-h_{m,m}I_{ s_m})$, then
 $$(\widetilde V_{1},\widetilde V_{2})=\mathtt{trunc}([\widebar V_1,\mathcal V_{1,1},\ldots,\mathcal V_{1,m}],\Theta_m,[\widebar V_2,\mathcal V_{2,1},\ldots,\mathcal V_{2,m}],\varepsilon_{\mathtt{orth}}),$$
 that is 
 \begin{equation}\label{eq:orth_rel_GS}
 \widetilde V_{1}\widetilde V_{2}^{T}+F_{1,m}F_{2,m}^{T}=\widebar V_1\widebar V_2^T -\sum_{j=1}^m h_{j,m}\mathcal{V}_{1,j}\mathcal{V}_{2,j}^T,
 \end{equation}
 where $F_{1,m}F_{2,m}^{T}$ is the matrix discarded during the application of Algorithm~\ref{trunc_routine1.1}.
 
 If $Q_{1}R_{1}=[\widebar V_1,\mathcal V_{1,1},\ldots,\mathcal V_{1,m}]$, $Q_{2}R_{2}=[\widebar V_2,\mathcal V_{2,1},\ldots,\mathcal V_{2,m}]$ denote the skinny QR factorizations performed during $\mathtt{trunc}$ and 
 $U\Sigma W^T=R_{1}\Theta_mR_{2}^T$ is the SVD with
$U=[u_1,\ldots,u_{\mathfrak{s}_m}]$, $W=[w_1,\ldots,w_{\mathfrak{s}_m}]$, 
 $\Sigma=\text{diag}(\sigma_1,\ldots,\sigma_{\mathfrak{s}_m})$, $\mathfrak{s}_m:=\widebar s+\sum_{j=1}^m s_j$, then we  consider the partitionings 
$$
U=[U_{k_m},\hat U],\quad W=[W_{k_m},\hat W],\quad \Sigma=\text{diag}(\Sigma_{k_m},\hat \Sigma),
$$
 where $U_{k_m}$, $W_{k_m}$, $\Sigma_{k_m}$ contain the leading $k_m$ singular vectors and, respectively, singular values, and $k_m$ is the smallest index such that $\sqrt{\sum_{i=k_m+1}^{\mathfrak{s}_m}\sigma_i^2}\leq\varepsilon_{\mathtt{orth}}\cdot\|\Sigma\|$.
We can write
\begin{align*}
\widetilde V_{1}&=Q_{1}U_{k_m}\Sigma_{k_m}^{\tfrac{1}{2}},\quad \widetilde V_{2}=Q_{2}W_{k_m}\Sigma_{k_m}^{\tfrac{1}{2}},\quad
%
%\widetilde V_{2}=Q_{2}\left([w_1,\ldots,w_{k_m}]\sqrt{\text{diag}(\sigma_1,\dots,\sigma_{k_m})}\right),
%\widetilde V_{1}=Q_{1}\left([u_1,\ldots, u_{k_m}]\sqrt{\text{diag}(\sigma_1,\dots,\sigma_{k_m})}\right),\quad
%\widetilde V_{2}=Q_{2}\left([w_1,\ldots,w_{k_m}]\sqrt{\text{diag}(\sigma_1,\dots,\sigma_{k_m})}\right),
%$$
%%
%and
%%
%{\small
%$$
  F_{1,m}=Q_{1}\hat U\hat\Sigma^{\tfrac{1}{2}},\quad  F_{2,m}=Q_{2}\hat W\hat\Sigma^{\tfrac{1}{2}},
%$$
\end{align*}
%$$
  %F_{1,m}=Q_{1}\left([u_{k_m+1},\ldots, u_{\mathfrak{s}_m}]\sqrt{\text{diag}(\sigma_{k_m+1},\dots,\sigma_{\mathfrak{s}_m})}\right),\;
  %F_{2,m}=Q_{2}\left([w_{k_m+1},\ldots, w_{\mathfrak{s}_m}]\sqrt{\text{diag}(\sigma_{k_m+1},\dots,\sigma_{\mathfrak{s}_m})}\right),
%$$
%}
%% 
%We first observe that
%$\langle\widetilde V_1\widetilde V_2^T,F_{1,m}F_{2,m}^T\rangle_F=0$. Indeed, 
and, 
since $Q_1$, $Q_2$, $W$ and $U$ are orthogonal matrices, %we have
\begin{align*}
%\langle \widetilde V_1\widetilde V_2^T,F_{1,m}F_{2,m}^{T}\rangle_F=&\text{trace}\left(
%\hat\Sigma U_{k_m}^T\hat U\hat\Sigma\hat W^T W_{k_m}\right)=0.
(Q_1U_{k_m})^TF_{1,m}F_{2,m}^T(W_{k_m}^TQ_2^T)^T\equiv 0.
 \end{align*}
 %%
 %\begin{flalign*}
%\langle \widetilde V_1\widetilde V_2^T,F_{1,m}F_{2,m}^{T}\rangle_F=&\text{trace}\left(
%\text{diag}(\sigma_1,\dots,\sigma_{k_m})[u_1,\ldots, u_{k_m}]^T[u_{k_m+1},\ldots, u_{\mathfrak{s}_m}] \right.\\
%& \quad\quad\;\;\;\text{diag}(\sigma_{k_m+1},\dots,\sigma_{\mathfrak{s}_m})\left.[w_{k_m+1},\ldots, w_{\mathfrak{s}_m}]^T [w_1,\ldots,w_{k_m}]\right)\\
%=&0.
 %\end{flalign*}
 %%
By pre and post-multiplying \eqref{eq:orth_rel_GS} by 
%$Q_1[u_1,\ldots,u_{k_m}](Q_1[u_1,\ldots,u_{k_m}])^T$ and $([w_1,\ldots,w_{k_m}]^TQ_2^T)^T[w_1,\ldots,w_{k_m}]^TQ_2^T$,
$Q_1U_{k_m}(Q_1U_{k_m})^T$ and $(W_{k_m}^TQ_2^T)^TW_{k_m}^TQ_2^T$, 
respectively, we thus get 
 %%
 %\begin{align*}
%\widetilde V_{1}\widetilde V_{2}^{T}=&Q_1[u_1,\ldots,u_{k_m}](Q_1[u_1,\ldots,u_{k_m}])^T\widebar V_1\widebar V_2^T([w_1,\ldots,w_{k_m}]^TQ_2^T)^T[w_1,\ldots,w_{k_m}]^TQ_2^T \\
%&-\sum_{j=1}^m h_{j,m}Q_1[u_1,\ldots,u_{k_m}](Q_1[u_1,\ldots,u_{k_m}])^T\mathcal{V}_{1,j}\mathcal{V}_{2,j}^T([w_1,\ldots,w_{k_m}]^TQ_2^T)^T[w_1,\ldots,w_{k_m}]^TQ_2^T.
 %\end{align*}
\begin{align*}
\widetilde V_{1}\widetilde V_{2}^{T}=&Q_1U_{k_m}(Q_1U_{k_m})^T\widebar V_1\widebar V_2^T(W_{k_m}^TQ_2^T)^TW_{k_m}^TQ_2^T \\
&-\sum_{j=1}^m h_{j,m}Q_1U_{k_m}(Q_1U_{k_m})^T\mathcal{V}_{1,j}\mathcal{V}_{2,j}^T(W_{k_m}^TQ_2^T)^TW_{k_m}^TQ_2^T.
 \end{align*}
 Since
\begin{align*}
h_{j,m}=&\langle \widebar V_1\widebar V_2^T,\mathcal{V}_{1,j}\mathcal{V}_{2,j}^T\rangle\\
=&\langle Q_1U_{k_m}(Q_1U_{k_m})^T\widebar V_1\widebar V_2^T(W_{k_m}^TQ_2^T)^TW_{k_m}^TQ_2^T,~~
Q_1U_{k_m}(Q_1U_{k_m})^T\mathcal{V}_{1,j}\mathcal{V}_{2,j}^T(W_{k_m}^TQ_2^T)^TW_{k_m}^TQ_2^T\rangle,
 \end{align*}
 %\begin{align*}
%h_{j,m}=&\langle \widebar V_1\widebar V_2^T,\mathcal{V}_{1,j}\mathcal{V}_{2,j}^T\rangle_F\\
%=&\langle Q_1[u_1,\ldots,u_{k_m}](Q_1[u_1,\ldots,u_{k_m}])^T\widebar V_1\widebar V_2^T([w_1,\ldots,w_{k_m}]^TQ_2^T)^T[w_1,\ldots,w_{k_m}]^TQ_2^T,\\
%&Q_1[u_1,\ldots,u_{k_m}](Q_1[u_1,\ldots,u_{k_m}])^T\mathcal{V}_{1,j}\mathcal{V}_{2,j}^T([w_1,\ldots,w_{k_m}]^TQ_2^T)^T[w_1,\ldots,w_{k_m}]^TQ_2^T\rangle_F,
 %\end{align*}
 %%
we can see $\widetilde V_{1}\widetilde V_{2}^{T}$ as the result of a specific Gram-Schmidt procedure in which %
%that orthogonolizes the projection
%\pnote{this reads weird. Can't we say .. as the result of a Gram-Schmidt orthogonalization of $\cdots$ onto Range($\cdots$) ?}
%of $\widebar V_1\widebar V_2^T$ onto $\text{Range}(Q_1[u_1,\ldots,u_{k_m}])\otimes\text{Range}(Q_2[w_1,\ldots,w_{k_m}])$, namely 
%of $\widebar V_1\widebar V_2^T$ onto $\text{Range}(Q_1U_{k_m})\otimes\text{Range}(Q_2W_{k_m})$. 
%In particular, 
%
%$$ Q_1[u_1,\ldots,u_{k_m}](Q_1[u_1,\ldots,u_{k_m}])^T\widebar V_1\widebar V_2^T([w_1,\ldots,w_{k_m}]^TQ_2^T)^T[w_1,\ldots,w_{k_m}]^TQ_2^T,$$
$$ Q_1U_{k_m}(Q_1U_{k_m})^T\widebar V_1\widebar V_2^T(W_{k_m}^TQ_2^T)^TW_{k_m}^TQ_2^T,$$
is orthogonalized against 
$Q_1U_{k_m}(Q_1U_{k_m})^T\mathcal{V}_{1,j}\mathcal{V}_{2,j}^T(W_{k_m}^TQ_2^T)^TW_{k_m}^TQ_2^T$ for all $j=1,\ldots,m$.

%with respect to 
%$Q_1[u_1,\ldots,u_{k_m}](Q_1[u_1,\ldots,u_{k_m}])^T\mathcal{V}_{1,j}\mathcal{V}_{2,j}^T([w_1,\ldots,w_{k_m}]^TQ_2^T)^T[w_1,\ldots,w_{k_m}]^TQ_2^T$ for all $j=1,\ldots,m$.

Moreover, each $\mathcal{V}_{1,j}\mathcal{V}_{2,j}^T$ can be written as 
\begin{align*}
 \mathcal{V}_{1,j}\mathcal{V}_{2,j}^T=& \mathcal{V}_{1,j}\mathcal{V}_{2,j}^T- Q_1U_{k_m}(Q_1U_{k_m})^T\mathcal{V}_{1,j}\mathcal{V}_{2,j}^T(W_{k_m}^TQ_2^T)^TW_{k_m}^TQ_2^T\\
 &+Q_1U_{k_m}(Q_1U_{k_m})^T\mathcal{V}_{1,j}\mathcal{V}_{2,j}^T(W_{k_m}^TQ_2^T)^TW_{k_m}^TQ_2^T,
\end{align*}
and, since 
$$\text{vec}\left(\mathcal{V}_{1,j}\mathcal{V}_{2,j}^T- Q_1U_{k_m}(Q_1U_{k_m})^T\mathcal{V}_{1,j}\mathcal{V}_{2,j}^T(W_{k_m}^TQ_2^T)^TW_{k_m}^TQ_2^T\right)\perp\text{Range}(Q_1U_{k_m})\otimes\text{Range}(Q_2W_{k_m}),
$$
we have 
$$\langle \mathcal{V}_{1,j}\mathcal{V}_{2,j}^T, \widetilde V_{1}\widetilde V_{2}^{T}\rangle_F=0, \quad \text{for all }j=1,\ldots,m.$$

To conclude, $\mathcal{V}_{1,m+1}\mathcal{V}_{2,m+1}^T=\widetilde{V}_{1}\widetilde{V}_{2}^{T}/\|\widetilde{V}_{1}\widetilde{V}_{2}^T\|$ and 
$\text{vec}(\mathcal{V}_{1,m+1}\mathcal{V}_{2,m+1}^T)$ has thus unit norm.
 \end{proof}

 As shown in the proof of Proposition~\ref{Prop_orthBasis}, to maintain the orthogonality of the basis, it is crucial that $\widetilde V_{1}\widetilde V_{2}^T$ and $F_{1,m}F_{2,m}^T$
  are block-orthogonal to each other, i.e., $(\widetilde V_{1}\widetilde V_{2}^T)^TF_{1,m}F_{2,m}^T=0$, see, e.g., \cite{Gutknecht2006}, an not only orthogonal with respect to the matrix inner product $\langle\cdot,\cdot\rangle$.
 This is due to the 
 QR-SVD-based truncation we perform. %A similar result can be shown also when Algorithm~\ref{trunc_routine2} is employed in the truncation steps. 
In general,
it may happen that the computed basis $V_{m+1}$ is no longer orthogonal if different truncation strategies are adopted. 
In this case, the theory developed in, e.g., \cite{Kandler2019} may be exploited to estimate the distance of the computed basis to orthogonality and such a value can be incorporated in the bounds \eqref{GMRES_bound_inexact1}-\eqref{GMRES_bound_inexact2}-\eqref{FOM_bound_inexact} to preserve the convergence of the overall iterative scheme.

In spite of Proposition~\ref{Prop_orthBasis}, in finite precision arithmetic the computed basis $V_{m+1}$ may fall short of being orthogonal and the employment of a modified Gram-Schmidt procedure with reorthogonalization -- as outlined in Algorithm~\ref{LR_FOM} -- is recommended. See, e.g., \cite{Giraud2005,Giraud2005a} for some discussions about the loss of orthogonality in the Gram-Schmidt procedure.

The truncations performed during the orthogonalization procedure consist in another source of inexactness that must be taken into account.
The inexact Arnoldi relation~\eqref{inexact_Arnoldi} becomes
%%
%{\small
\begin{equation*}\label{inexact_Arnoldi2}
 \mathcal{A}V_m-[\text{vec}(E_1),\ldots,\text{vec}(E_m)]=V_{m}H_m+h_{m+1,m}\text{vec}(\mathcal{V}_{1,m+1}\mathcal{V}_{2,m+1}^T)e_m^T+[\text{vec}(F_{1,1}F_{2,1}^{T}),\ldots,\text{vec}(F_{1,m}F_{2,m}^{T})],
\end{equation*}
%}
%%
and one can derive results similar to the ones in Theorem~\ref{Theorem_inexactProd} for the inexact Arnoldi relation
$$
 \mathcal{A}V_m-[\text{vec}(E_1+F_{1,1}F_{2,1}^{T}),\ldots,\text{vec}(E_m+F_{1,m}F_{2,m}^{T})]=V_{m}H_m+h_{m+1,m}\text{vec}(\mathcal{V}_{1,m+1}\mathcal{V}_{2,m+1}^T)e_m^T,
 $$
obtaining estimates for $\|E_k+ F_{1,k}F_{2,k}^{T}\|$. Since
$$\|E_k+ F_{1,k}F_{2,k}^{T}\|\leq \|E_k\|+\| F_{1,k}F_{2,k}^{T}\|,$$
it may be interesting to study how to distribute the allowed inexactness between the truncation steps.

Since the rank of the iterates grows less dramatically during the orthogonalization step compared to what happens after the multiplication
with $\mathcal{A}$, we allow $2\|E_k\|$ 
to grow in accordance with 
Theorem~\ref{Theorem_inexactProd}, while $\| F_{1,k}F_{2,k}^{T}\|$ is maintained sufficiently small.
Indeed, the matrix $[\widebar V_1,\mathcal{V}_{1,1},\ldots,\mathcal{V}_{1,m}]\Theta_m[\widebar V_2,\mathcal{V}_{2,1}\ldots,\mathcal{V}_{2,m}]^T$ in line~\ref{algo_secondtrunc} of Algorithm~\ref{LR_FOM}
is, in general, very rank-deficient and a significant reduction in the number of columns to be stored takes place even when the $\mathtt{trunc}$ function is applied with a small threshold.

In particular, at the $m$-th iteration, we can set 
\begin{equation}\label{trunctol_orth}
\varepsilon_{\mathtt{orth}}=\min\{\|E_k\|,\varepsilon/( m_{\max})\},
\end{equation}
 where $\varepsilon$ is the desired accuracy of the final solution in terms of 
relative residual norm.
This means that $\|E_k+ F_{1,k}F_{2,k}^{T}\|$ fulfills the estimates in~\eqref{GMRES_bound_inexact1}-\eqref{GMRES_bound_inexact2}-\eqref{FOM_bound_inexact} and the convergence is 
thus preserved.  

The vectors $y_m^{fom}$ and $y_m^{gm}$ can be still computed as in~\eqref{FOM_linearsystem}-\eqref{least_squares_GMRES} and Proposition~\ref{Prop_constraints} holds also when the low-rank truncation in line~\ref{algo_secondtrunc} of Algorithm~\ref{LR_FOM} are performed.
\begin{Prop}\label{Prop_constraints2}
Let~\eqref{inexact_Arnoldi2} hold and define $W_m=\mathcal{A}V_m-[\text{vec}(E_1),\ldots,\text{vec}(E_m)]$. If $y_m^{gm}$ is computed as in~\eqref{least_squares_GMRES}, where $\underline{H}_m$ stems from the low-rank Arnoldi procedure illustrated in Algorithm~\ref{LR_FOM} with low-rank truncations are performed by~Algorithm~\ref{trunc_routine1.1}, then $q_m^{gm}:=W_my_m^{gm}$ is such that
$$ q_m=\argmin_{q\in\text{Range}(W_m)}\|\text{vec}(C_1C_2^T)+q\|.$$
Similarly, if $y_m^{fom}$ is computed as in~\eqref{FOM_linearsystem} where $H_m$ is the principal square submatrix of the aforementioned $\underline{H}_m$, then $q_m^{fom}:=W_my_m^{fom}$ is such that
$$\text{vec}(C_1C_2^T)+q_m\perp \text{Range}(V_m).$$
\end{Prop}
\begin{proof}
 We only need to prove that $V_m^T[\text{vec}(F_{1,1}F_{2,1}^{T}),\ldots,
 \text{vec}(F_{1,m}F_{2,m}^{T})]=0$
 as the rest of the proof comes from 
~\cite[Proposition 3.2-3.3]{Simoncini2003}.
 
 Using the same arguments of the proof of Proposition~\ref{Prop_orthBasis}, we can show that $F_{1,j}F_{2,j}^{T}$ is orthogonal to $\mathcal{V}_{1,i}\mathcal{V}_{2,i}^T$ for all  $j,i=1,\ldots,m$, $i+1\neq j$. Therefore, the only nonzero components of 
$$V_m^T[\text{vec}(F_{1,1}F_{2,1}^{T}),\ldots,
 \text{vec}(F_{1,m}F_{2,m}^{T})],$$
 are in the first subdiagonal. These entries are of the form  $\langle \mathcal{V}_{1,\ell+1}\mathcal{V}_{2,\ell+1}^T,F_{1,\ell}F_{2,\ell}^{T}\rangle_F$ and we show they are zero for every $\ell=1,\ldots,m-1$. We have $\mathcal{V}_{1,\ell+1}\mathcal{V}_{2,\ell+1}^T=\widetilde{V}_{1}\widetilde{V}_{2}^{T}/\|\widetilde{V}_{1}\widetilde{V}_{2}^T\|$ and in the proof of Proposition~\ref{Prop_orthBasis} we have already shown that 
 $\langle \widetilde{V}_{1}\widetilde{V}_{2}^T,F_{1,\ell}F_{2,\ell}^T\rangle_F=0$. This completes the proof.
\end{proof}

%A similar result can be shown when Algorithm~\ref{trunc_routine2} is used in place of Algorithm~\ref{trunc_routine1.1} as well.

The true relative residual norm can be written as 
$$r_m=\widetilde r_m-[\text{vec}(E_1+F_{1,1}F_{2,1}^{T}),\ldots,\text{vec}(E_m+F_{1,m}F_{2,m}^{T})]y_m,$$
 and following the discussion in Remark~\ref{Remark_residualnorm} we have
 \begin{equation}\label{upperbound_res2}
\|r_m\|\leq \|\widetilde r_m\|+\sum_{j=1}^m \|E_j\|\cdot|e_j^Ty_m|+\sum_{j=1}^m \|F_{1,j}F_{2,j}^{T}\|\cdot|e_j^Ty_m|
\leq\|\widetilde r_m\|+\sum_{j=1}^m \left(\varepsilon_{\mathcal{A}}^{(j)}+\frac{m}{m_{\max}}\varepsilon\right)|e_j^Ty_m|,
%+\frac{m}{m_{\max}}\varepsilon\sum_{j=1}^m |e_j^Ty_m|,  
 \end{equation}
so that the right-hand side in the above expression must be computed to check convergence.
% \pnote{just a thought: Would it be interesting to plug $\varepsilon_{\mathcal{A}}^{(j)}$ from (3.6),(3.7) into \eqref{upperbound_res2}??.. Does this give some cool insight, e.g.:}
% $$
% \|r_m\|\leq\|\widetilde r_m\|+\sum_{j=1}^m \left(c_1{\|\widetilde r_{j-1}^{gm}\|}+m\right)\frac{\varepsilon}{m_{\max}}|e_j^Ty_m|,
% $$
% \pnote{... I guess not}
%%%%%%%%%%%%%%%%%%%%%%%%%%%%%%%%%%%%%%%%%%%%%%%%%%

\section{Alternative truncation strategies}\label{sec:ranktrunc}

As we discussed above, to keep the low-rank Krylov methods computationally feasible, 
the quantities involved in the solution process have to be compressed so that their rank, i.e., the sizes of the low-rank factors, is kept small. Let $NML^T$ with factors $N,L\in\mathbb{R}^{n\times m}$, $M\in\mathbb{R}^{m\times m}$, be the quantity to be compressed, and
assume that $\mathrm{rank}(NML^T)=m$.
So far we have used a direct approach using QR and SVD decompositions in Algorithm~\ref{trunc_routine1.1} which essentially computes a partial SVD of $NML^T$ corresponding to all $m$ nonzero singular values. This whole procedure relies heavily on dense linear algebra computations and can, hence, become quite expensive. This is especially due to the QR decompositions which will be expensive if the rectangular factors $N,L$ have many columns. Moreover, if $NML^T$ has a very small numerical numerical rank, say $k\ll m$, then Algorithm~\ref{trunc_routine1.1} will generate a substantial computational overhead because $m-k$ singular vectors will be thrown away. Nevertheless, thanks to the complete knowledge of all singular values, this procedure is able to correctly assess the truncation error in the Frobenius norm so that the required accuracy of the truncation is always met.

Following the discussion in, e.g.,~\cite{Stoll2015,BenOnwSto15,Onw16}, a more economical alternative could be to compute only a partial SVD $NML^T\approx U_k\Sigma_kW_k^T$ associated to the $k$ singular values that are larger than the given truncation threshold. If also the $(k+1)$-th singular value is computed, one has the truncation error in the 2-norm: $\|NML^T-U_k\Sigma_kW_k^T\|_2\leq \sigma_{k+1}(NML^T)$. Obviously, the results of the previous section are still valid if this form of truncation is used. 
Approximations of the dominant singular values and corresponding singular vectors can be computed by iterative methods for large-scale SVD computations as, e.g., Lanczos bidiagonalization~(see, e.g., \cite{Lar98,BagRei05,Sto12}) or Jacobi-Davidson methods; see \cite{Hoc01}.
To apply these methods, only matrix vector products $N(M(L^Tx))$ and $L(M^T(N^Tx))$ are required. For achieving the compression goal one could, e.g., compute $k_{\max}\geq k$ triplets and, if required, neglect any singular vectors corresponding to singular value below a certain threshold. However, we do in general not know in advance how many singular values will be larger than a given threshold. Picking a too small value of $k_{\max}$ can lead to very inaccurate truncations that do not satisfy the required thresholds \eqref{GMRES_bound_inexact1}--\eqref{FOM_bound_inexact}, \eqref{trunctol_orth} and, therefore, endanger the convergence of the low-rank Krylov method.  
Some of aforementioned iterative SVD methods converge theoretically monotonically, i.e., the singular values are found in a decreasing sequence starting with the largest one. Hence, the singular value finding iteration can be kept running until a sufficiently small singular value approximation, e.g., $\tilde \sigma< \varepsilon_{\text{trunc}}\|NML^T\|_2$, is detected. 
In the practical situations within low-rank Krylov methods, the necessary number of singular triplets can be $\mathcal{O}(10^2)$ or larger and it may be difficult to ensure that the iterative SVD algorithms do not miss some of the largest singular values or that no singular values are detected several times. Due to the sheer number of occurrences where compression is required in Algorithm~\ref{LR_FOM}, preliminary tests with iterative SVD methods did not yield any substantial savings compared to the standard approach in Algorithm~\ref{trunc_routine1.1}. 

Compression algorithms based on randomized linear algebra might offer further alternative approaches with reduced computational times. See, e.g., \cite{HalMT11,KreP17,Che2019}.  
%%%%%%%%%%%%%%%%%%%%%%%%%%%%%%%%%%%%%%%%%%%%%%%%%%
\section{Preconditioning}\label{Preconditioning}

It is well-known that Krylov methods require preconditioning in order to obtain a fast convergence in terms of number of iterations and low-rank Krylov methods are no exception. However, due to the peculiarity of our framework, the preconditioner operator must possess some supplementary features with respect to standard preconditioners for linear systems. Indeed, in addition to be effective in reducing the number of iterations at a reasonable computational cost, the preconditioner operator must not dramatically increase 
the memory requirements of the solution process.

Given a nonsingular operator $\mathcal{P}$ or its inverse $\mathcal{P}^{-1}$, if we employ right preconditioning, the original systems~\eqref{eq:linear_system} is transformed into
\begin{equation}\label{eq:precond_linearsystem}
 \mathcal{A}\mathcal{P}^{-1}\widebar x=-\text{vec}(C_1C_2^T),\quad \text{vec}(X)=\mathcal{P}^{-1}\widebar x,
\end{equation}
so that, at each iteration $m$, we have to apply 
$\mathcal{P}^{-1}$ to the current basis vector $\text{vec}(\mathcal{V}_{1,m}\mathcal{V}_{2,m}^T)$. Note that we restrict ourselves here to right preconditioning because this has the advantage that one can still monitor the true unpreconditioned residuals without extra work within the Krylov routine. Of course, in principle also left and two-sided preconditioning can be used.

The preconditioning operation must be able to exploit the low-rank format of $\mathcal{V}_{1,m}\mathcal{V}_{2,m}^T$. Therefore, a naive operation of the form $\mathcal{P}^{-1}\text{vec}(\mathcal{V}_{1,m}\mathcal{V}_{2,m}^T)$ is not admissible in our context as this would require the allocation of the dense $n\times n$ matrix $\mathcal{V}_{1,m}\mathcal{V}_{2,m}^T$. One way to overcome this numerical difficulty is to employ a preconditioner operator $\mathcal{P}$ which allows for a representation in terms of a Kronecker sum, namely
\begin{equation}\label{precond_expression}
\mathcal{P}=\sum_{i=1}^\ell P_i\otimes 
T_i. 
\end{equation}
This means that the operation $z_m=\mathcal{P}^{-1}\text{vec}(\mathcal{V}_{1,m}\mathcal{V}_{2,m}^T)$ is equivalent to solving the matrix equation 
\begin{equation}\label{precond_equation}
\sum_{i=1}^\ell T_iY_mP_i^T-\mathcal{V}_{1,m}\mathcal{V}_{2,m}^T=0, \quad \text{vec}(Y_m)=z_m. 
\end{equation}

In our setting,
the operator $\mathcal{P}$ often amounts to an approximation to $\mathcal{A}$ in~\eqref{eq:linear_system} obtained by either dropping some terms in the series or replacing some of them by a multiple of the identity. See, e.g.,~\cite{Palitta2016,Powell2017,Ullmann2010}. 
Another option that has not been fully explored in the matrix equation literature so far is the case of polynomial preconditioners~(see, e.g., \cite{Gij95,LiuMW15}) where $\mathcal{P}^{-1}$ resembles a fixed low-degree polynomial evaluated in $\mathcal{A}$.
Alternatively, we can formally set $\mathcal{P}=\mathcal{A}$ in~\eqref{precond_expression} and inexactly solve equation~\eqref{precond_equation} by few iterations of another Krylov method (e.g., Algorithm~\ref{LR_FOM}) leading to an inner-outer Krylov method; see, e.g., \cite{Simoncini2002}. 

Clearly, equation~\eqref{precond_equation} must be easy to solve.
For instance, if $\ell=1$, then $Y_m=(T_1^{-1}\mathcal{V}_{1,m})(P_1^{-1}\mathcal{V}_{2,m})^T$ and an exact application of the preconditioner can be carried out. Similarly, when $\ell=2$ and a fixed number of ADI iterations are performed at each Krylov iteration $m$, then it is easy to show that we are still working in an exact preconditioning framework. See, e.g. \cite{Damm2008,Benner2013}.
In all these cases, the results presented in the previous sections still hold provided $\mathcal{A}$ is replaced by the preconditioned matrix $\mathcal{AP}^{-1}$.

Equation~\eqref{precond_equation} is often iteratively solved and, in general, this procedure leads to the computation of a low-rank approximation $\mathcal{Z}_{1,m}\mathcal{Z}_{2,m}^T$ to $Y_m$ that has to be interpreted as a variable preconditioning scheme with a different preconditioning operator at each outer iteration. In this cases, a flexible variant of Algorithm~\ref{LR_FOM} must be employed which consists in a standard flexible Krylov procedure equipped with the low-rank truncations presented in the previous sections. See, e.g.,~\cite[Section 10]{Simoncini2007a} for some details about flexible Krylov methods and~\cite[Section 9.4.1]{Saad2003,Saad1993} for a discussion about flexible GMRES. 

We must mention that the employment of a flexible procedure doubles, at least, the memory requirements of the solution process. Indeed, both the \emph{preconditioned} and \emph{unpreconditioned} bases must be stored and $\text{rank}(\mathcal{Z}_{1,m}\mathcal{Z}_{2,m}^T)\geq \text{rank}(\mathcal{V}_{1,m}\mathcal{V}_{2,m}^T)$ for all $m$. This aspect must be taken into account when designing the preconditioner. See Example~\ref{Ex.1}.

At a first glance, the presence of a variable preconditioning 
procedure can complicate the derivations illustrated in sections~\ref{Inexact matrix-vector products}-\ref{Structured perturbations of the basis} for the safe selection of the low-rank truncation thresholds that guarantee the convergence of the solution method.
Indeed, if at iteration $m$, $\mathcal{Z}_{1,m}\mathcal{Z}_{2,m}^T$ is the result of the preconditioning step~\eqref{precond_equation}, 
we still want to truncate the matrix $[A_1\mathcal{Z}_{1,m},\ldots, A_p\mathcal{Z}_{1,m}][B_1\mathcal{Z}_{2,m},\ldots, B_p\mathcal{Z}_{2,m}]^T$ in order to moderate the storage demand and one may wonder if the inexactness of step~\eqref{precond_equation} plays a role in such a truncation. Thanks to the employment of a flexible strategy, we are going to show how the tolerances for the low-rank truncations, namely $\varepsilon_{\mathcal{A}}$ and $\varepsilon_{\mathtt{orth}}$ in Algorithm~\ref{LR_FOM}, can be still computed as illustrated in sections~\ref{Inexact matrix-vector products}-\ref{Structured perturbations of the basis}.

Flexible Krylov methods are characterized not only by having a preconditioner that changes at each iteration, but also from the fact that the solution is recovered by means of the preconditioned basis. In particular,
$$\text{vec}(X_m)=Z_my_m,\quad Z_m:=[\text{vec}(\mathcal{Z}_{1,1}\mathcal{Z}_{2,1}^T),\ldots,\text{vec}(\mathcal{Z}_{1,m}\mathcal{Z}_{2,m}^T)],$$
see, e.g.,~\cite{Saad1993}; this is a key ingredient in our analysis.

We start our discussion by considering flexible Krylov methods with no truncations. For this class of solvers the relation 
\begin{equation}\label{flexible_notruncation}
 \mathcal{A}Z_m=V_mH_m+h_{m+1,m}\text{vec}(\mathcal{V}_{1,m+1}\mathcal{V}_{2,m+1}^T)e_m^T,
\end{equation}
holds, see, e.g.,~\cite[Equation (9.22)]{Saad2003}, and $\text{span}\{\text{vec}(\mathcal{Z}_{1,1}\mathcal{Z}_{2,1}^T),\ldots,\text{vec}(\mathcal{Z}_{1,m}\mathcal{Z}_{2,m}^T)\}$ is not a Krylov subspace in general.
Therefore, also for the flexible Krylov methods with no low-rank truncations we must consider constrains different from the ones in~\eqref{Galerkin}-\eqref{MresGalerkin} and results similar to the ones in Proposition~\ref{Prop_constraints} with $W_m=\mathcal{A}Z_m$ hold. See, e.g.,~\cite[Proposition 9.2]{Saad2003}.

If we now introduce a low-rank truncation of the matrix 
$$[A_1\mathcal{Z}_{1,m},\ldots, A_p\mathcal{Z}_{1,m}][B_1\mathcal{Z}_{2,m},\ldots, B_p\mathcal{Z}_{2,m}]^T,$$ 
at each iteration $m$, that is we compute 
\begin{equation}\label{truncation_flexible}
(\widebar V_1, \widebar V_2)=\mathtt{trunc}([A_1\mathcal{Z}_{1,m},\ldots, A_p\mathcal{Z}_{1,m}],I,[B_1\mathcal{Z}_{2,m},\ldots, B_p\mathcal{Z}_{2,m}],\varepsilon_{\mathcal{A}}), 
\end{equation}
then the relation~\eqref{flexible_notruncation} becomes
\begin{equation}\label{flexible_onetruncation}
 \mathcal{A}Z_m-[\text{vec}(E_1),\ldots,\text{vec}(E_m)]=V_mH_m+h_{m+1,m}\text{vec}(\mathcal{V}_{1,m+1}\mathcal{V}_{2,m+1}^T)e_m^T,
\end{equation}
where the matrices $E_k$'s are the ones discarded when~\eqref{truncation_flexible} is performed.
If $\|E_k\|$ satisfies the inequalities in Theorem~\ref{Theorem_inexactProd}, then the convergence of the low-rank flexible Krylov procedure is still guaranteed in the sense that the residual norm keeps decreasing as long as 
$\text{span}\{\text{vec}(\mathcal{Z}_{1,1}\mathcal{Z}_{2,1}^T),\ldots,\text{vec}(\mathcal{Z}_{1,m}\mathcal{Z}_{2,m}^T)\}$ grows. However, the matrix $H_m$ no longer represents an approximation of $\mathcal{A}$ onto the current subspace and the approximation of $\sigma_{m_{\max}}(\underline{H}_{m_{\max}})$ and $\sigma_1(\underline{H}_{m_{\max}})$ in the right-hand side of~\eqref{GMRES_bound_inexact1}-\eqref{GMRES_bound_inexact2}-\eqref{FOM_bound_inexact} by the corresponding singular values of $\mathcal{A}$ may no longer be effective. In our numerical experience, approximating
$\sigma_{m_{\max}}(\underline{H}_{m_{\max}})$ and $\sigma_1(\underline{H}_{m_{\max}})$ by the smallest and largest singular values of the preconditioned matrix $\mathcal{A}\mathcal{P}^{-1}$, i.e., mimicking what is done in case of exact applications of $\mathcal{P}$,
provides satisfactory results.
Obtaining computable approximations to $\sigma_{m_{\max}}(\underline{H}_{m_{\max}})$ and $\sigma_1(\underline{H}_{m_{\max}})$ for the inner-outer approach is not straightforward. In this case, a practical approach may be to still approximate
$\sigma_{m_{\max}}(\underline{H}_{m_{\max}})$ and $\sigma_1(\underline{H}_{m_{\max}})$
by $\sigma_{n^2}(\mathcal{A})$ and $\sigma_{1}(\mathcal{A})$, respectively.
%Indeed, assuming that $\sigma_{m_{\max}}(\underline{H}_{m_{\max}})\approx \sigma_{n^2}(\mathcal{A}\mathcal{P}^{-1})$, we have $\sigma_{n^2}(\mathcal{A}\mathcal{P}^{-1})=\sqrt{\lambda_{n^2}(\mathcal{P}^{-T}\mathcal{A}^T\mathcal{A}\mathcal{P}^{-1})}$ and the spectrum of $\mathcal{P}^{-T}\mathcal{A}^T\mathcal{A}\mathcal{P}^{-1}$ is contained in the spectral region of $\mathcal{A}^T\mathcal{A}$. 
%\pnote{Wait, what do you mean by \textit{contained}? The spectrum of $\mathcal{P}^{-T}\mathcal{A}^T\mathcal{A}\mathcal{P}^{-1}$ can be bigger than $\mathcal{A}^T\mathcal{A}$.}
 These approximations may be very rough as they completely neglect the role of the preconditioner so that they may lead to quite conservative truncation thresholds. However, at the moment, we do not see any another possible alternatives.

The introduction of the low-rank truncations that lead to~\eqref{flexible_onetruncation} implies that the constrained imposed on the residual vector are no longer in terms of the space spanned by $Z_m$ and the results presented in Proposition~\ref{Prop_constraints} with $W_m=\mathcal{A}Z_m-[\text{vec}(E_1),\ldots,\text{vec}(E_m)]$ hold.

In flexible Krylov methods, the orthogonalization procedure involves only the unpreconditioned basis $V_m$ so that  the truncation step in line~\ref{algo_secondtrunc} of Algorithm~\ref{LR_FOM}
is not really affected by the preconditioning procedure and the results in Proposition~\ref{Prop_orthBasis}-\ref{Prop_constraints2} are still valid. The truncation threshold $\varepsilon_{\mathtt{orth}}$ can be still selected as proposed in section~\ref{Structured perturbations of the basis}.

%%%%%%%%%%%%%%%%%%%%%%%%%%%%%%%%%%%%%%%%%%%%%%%%%%%%%%%%%%%%%%%%%%%%%%%%%%
\section{Short recurrence methods}\label{Short recurrence methods}
Short recurrence Krylov methods can be very appealing in our context as only a fixed, usually small, number of basis vectors have to be stored.
In case of symmetric problems, i.e., equation~\eqref{eq_main} where all the coefficient matrices $A_i$'s and $B_i$'s are symmetric, 
the low-rank MINRES algorithm proposed in~\cite{Paige1975} can be employed in  the solution process. 

If $\mathcal{A}$ in~\eqref{eq:linear_system} is also positive definite, the low-rank CG method~illustrated in \cite{Hestenes1952} is a valid candidate for the solution of equation~\eqref{eq_main}. Notice that, in general, it is not easy to characterize the spectral distribution of $\mathcal{A}$ in terms of the spectrum of the coefficient matrices $A_i$'s and $B_i$'s. However, it can be shown that if $A_i$ and $B_i$ are positive definite for all $i$, then also $\mathcal{A}$ is positive definite.

Short recurrence methods can be appealing also in case of a nonsymmetric $\mathcal{A}$ and low-rank variants of BICGSTAB~\cite{Vorst1992}, QMR~\cite{Freund1991} or other methods can be employed to solve equation~\eqref{eq_main}.

See, e.g.,~\cite{Stoll2015,Benner2013} for an implementation of low-rank MINRES, CG and BICGSTAB. 

In all the short recurrence Krylov methods, the constructed basis $V_m$ is not orthogonal in practice and this loss of orthogonality must be taken into account in the bounds for the allowed inexactness proposed in Theorem~\ref{Theorem_inexactProd}. In~\cite[Section 6]{Simoncini2003}, the authors propose to incorporate the smallest singular values of the computed basis, namely $\sigma_m(V_m)$, in the right-hand side of~\eqref{GMRES_bound_inexact1}-\eqref{GMRES_bound_inexact2}-\eqref{FOM_bound_inexact} to guarantee the convergence of the method.
However, no practical approximation to $\sigma_m(V_m)$ is proposed in~\cite{Simoncini2003}.

A different approach that can be pursued is the one illustrated in~\cite{Bouras2005}. 
In this paper the authors propose to select bounds of the form 
\begin{equation}\label{eq_CERFACS}
\|E_k\|\leq \min\left\{\alpha_k\varepsilon,1\right\}, \quad \alpha_k=\frac{1}{\min\left\{\|\widetilde r_k\|,1\right\}},
\end{equation}
where $\widetilde r_k$ is the current computed residual vector, and in~\cite{Eshof2004} the authors studied the effects of such a choice on
the convergence of a certain class of inexact Krylov methods. In particular, in~\cite{Eshof2004} it is shown how the residual gap $\delta_m$ remains small if $\|E_k\|$ fulfills~\eqref{eq_CERFACS} for all $k\leq m$. Even though the true residual and the computed one are close, this does not imply that the residual norm is actually always small and we thus have to assume that the norm of the computed residual goes to zero as it is done in~\cite{Eshof2004}.
%%%%%%%%%%%%%%%%%%%%%%%%%%%%%%%%%%%%%%%%%%%%%%%%%
\section{Numerical examples}\label{Numerical examples}

In this section we present some numerical results that confirm the theoretical analysis derived in the previous sections.
To this end we consider some general multiterm linear matrix equation of the form~\eqref{eq_main} stemming from the discretization of certain deterministic and stochastic PDEs.

We apply the LR-GMRES variant of Algorithm~\ref{LR_FOM} in the solution process and we always select Algorithm~\ref{trunc_routine1.1} for the low-rank truncations. 

We report the number of performed iterations, the rank of the computed solution, the computational time needed to calculate such a solution together with the relative residual norm achieved, and the storage demand. For the latter, we 
document the number of columns $\mathfrak{s}=\sum_{j=1}^{m+1}s_j$ of the matrix $[\mathcal{V}_{1,1},\ldots,\mathcal{V}_{1,m+1}]$, where $m$ is the number of iterations needed to converge.
Similarly, if a flexible strategy is adopted, we also report the number of columns $\mathfrak{z}$ of
$[\mathcal{Z}_{1,1},\ldots,\mathcal{Z}_{1,m}]$.

This means that, for equations of the form~\eqref{eq_main} where $n_A=n_B=n$, we have to allocate $2\mathfrak{s}$ ($2(\mathfrak{s}+\mathfrak{z})$) vectors of length $n$. If $n_A\neq n_B$, the memory requirements amount to $\mathfrak{s}$ ($\mathfrak{s}+\mathfrak{z}$) vectors of length $n_A$ and $\mathfrak{s}$ ($\mathfrak{s}+\mathfrak{z}$) vectors of length $n_B$.

The solution process is stopped as soon as the upper bound on the residual norm in~\eqref{upperbound_res2}, normalized by $\|C_1C_2^T\|_F$, gets smaller than $10^{-6}$.

As already mentioned, we always assume that the exact solution $X$ admits accurate low-rank approximations. Nevertheless, if $S_1,S_2$ are the low-rank factors computed by Algorithm~\ref{LR_FOM}, we report also the real relative residual norm 
$\|\sum_{i=1}^pA_iS_1S_2^TB_i^T+C_1C_2^T\|_F/\|C_1C_2^T\|_F$ in the following to confirm the 	reliability of our numerical procedure. Once again, the real residual norm can be computed at low cost by exploiting the low rank of $S_1S_2^T$ and the cyclic property of the trace operator.

All results were obtained with Matlab R2017b~\cite{MATLAB} on a Dell machine with 2.4GHz
processors and 250 GB of RAM.

\begin{num_example}\label{Ex.1}
{\rm
We consider a slight modification of Example 4 in~\cite{Palitta2016}. In particular, the continuous problem we have in mind is the convection-diffusion equation 
\begin{equation}\label{Ex.1_eq}
 \begin{array}{rlll}
         -\nu\Delta u + \vec{w}\cdot\nabla u&=&1,&  \text{in }D=(0,1)^2,\\
         u&=&0,&\text{on }\partial D,\\
        \end{array}
\end{equation}
where $\nu>0$ is the viscosity parameter and the convection vector $\vec{w}$ is given by $\vec{w}=(\phi_1(x)\psi_1(y),\phi_2(x)\psi_2(y))=((1-(2x+1)^2)y,-2(2x+1)(1-y^2))$. 
%\pnote{maybe it would be better for notation to not use $\nu$  here}
The centered finite differences discretization of equation~\eqref{Ex.1_eq} yields the following matrix equation
\begin{equation}\label{Ex.1_eq2}
\nu TX+\nu XT+\Phi_1BX\Psi_1+\Phi_2XB^T\Psi_2-\mathbf{1}\mathbf{1}^T=0, 
\end{equation}
where $T\in\mathbb{R}^{n\times n}$ is the negative discrete laplacian, $B\in\mathbb{R}^{n\times n}$ corresponds to the discretization of the first derivative, $\Phi_i$ and $\Psi_i$ are diagonal matrices collecting the nodal values of the corresponding functions $\phi_i$, $\psi_i$, $i=1,2$, and $\mathbf{1}\in\mathbb{R}^n$ is the vector of all ones. See~\cite{Palitta2016} for more details.

Even though equation~\eqref{Ex.1_eq2} amounts to a generalized Sylvester equation, the  solution schemes available in the literature and tailored to this kind of problems cannot be applied to equation~\eqref{Ex.1_eq2} in general. Indeed, to the best of our knowledge, all the existing methods for large-scale generalized equations rely on a splitting of the overall discrete operator of the form $\mathcal{M}+\mathcal{N}$, $\mathcal{M}(X)=\nu TX+\nu XT$, $\mathcal{N}(X)=\Phi_1BX\Psi_1+\Phi_2XB^T\Psi_2$, which is supposed to be convergent. See, e.g., \cite{Jarlebring2018,Shank2015,Benner2013}.
However, the latter property may be difficult to meet in case of the convection-diffusion equation, especially for dominant convection.

%\pnote{This is not quite true: \cite{Kressner2015} doesn't require a convergent splitting}
We thus have to interpret \eqref{Ex.1_eq2} as a general multiterm matrix equation of the form~\eqref{eq_main} and we solve it by the preconditioned LR-GMRES. Following the discussion in~\cite{Palitta2016}, we use the operator
$$
\begin{array}{rrll}
\mathcal{L}:& \mathbb{R}^{n\times n}&\rightarrow& \mathbb{R}^{n\times n}\\
&X&\mapsto& (\nu T+\widebar\psi_1\Psi_1B)X+X(\nu T+\widebar\phi_2B^T\Psi_2),
\end{array}$$
as preconditioner, where $\widebar\psi_1,\widebar\phi_2\in\mathbb{R}$ are the mean values  of $\psi_1(y)$ and $\phi_2(x)$ on $(0,1)$, respectively.

At each LR-GMRES iteration, we approximately invert $\mathcal{L}$ by performing 10 iterations of the extended Krylov subspace method for Sylvester equation\footnote{A Matlab implementation is available at {\tt http://www.dm.unibo.it/\textasciitilde simoncin/software.html}.} derived in~\cite{Breiten2016}. Since this scheme gives a different preconditioner every time it is called, we must employ the flexible variant of LR-GMRES.
To avoid an excessive increment in the memory requirements due to the allocation of both the preconditioned and unpreconditioned bases, we do not apply $\mathcal{L}$ to the current basis vector, i.e., at iteration $k$, we do not compute $\mathcal{Z}_{1,k}\mathcal{Z}_{2,k}^T\approx\mathcal{L}^{-1}(\mathcal{V}_{1,k}\mathcal{V}_{2,k}^T)$. We first truncate the low-rank factors $\mathcal{V}_{1,k},\mathcal{V}_{2,k}$, namely we compute $(\mathcal{\widehat V}_{1,k},\mathcal{\widehat V}_{2,k})=\mathtt{trunc}(\mathcal{V}_{1,k},I,\mathcal{V}_{2,k},\varepsilon_{\mathtt{precond}})$, 
and then define $\mathcal{Z}_{1,k},\mathcal{Z}_{2,k}$ such that $\mathcal{Z}_{1,k}\mathcal{Z}_{2,k}^T\approx\mathcal{L}^{-1}(\mathcal{\widehat V}_{1,k}\mathcal{\widehat V}_{2,k}^T)$. This procedure leads to a lower storage demand of the overall solution process and to less time consuming preconditioning steps. On the other hand, the effectiveness of the  preconditioner in reducing the total iteration count may get weakened, especially for large $\varepsilon_{\mathtt{precond}}$. In the results reported in the following we have always set $\varepsilon_{\mathtt{precond}}=10^{-3}$.

In Table~\ref{tab1.1} we report the results for different values of $n$ and $\nu$.

\begin{table}[!ht]
 \centering
 %{\footnotesize
 \caption{Example~\ref{Ex.1}. Results for different values of $n$ and $\nu$.\label{tab1.1}}
 \setlength\tabcolsep{8pt}
\begin{tabular}{|r  r| r r r| r r |r r | }
\hline
  &  &  &  & &\multicolumn{2}{c|}{Memory} & \multicolumn{2}{c|}{Conv. Checks}  \\
  $\nu$ &  $n$ & It. & $\text{rank}(S_1S_2^T)$ & Time (s) & $V_m$ & $Z_m$ &\eqref{upperbound_res2}$/\|C_1C_2^T\|_F$ & Real Res.  \\
 \hline
 
0.5 & 5000 & 8 & 58 & 2.872e1 & 1174  & 915 & 4.078e-7 & 2.974e-7 \\

 & 10000 & 8 & 59 & 8.352e1 & 1543 & 1079 & 4.242e-7 & 3.144e-7 \\

 & 15000 & 8 & 69 & 1.812e2 & 2075 & 1239 & 9.492e-7 & 6.401e-7 \\
\hline
0.1 & 5000 & 15 & 66 & 1.256e2 & 3284 & 1880 & 7.803e-7 & 4.509e-7 \\

 & 10000 & 15 & 71 & 4.687e2 & 4566 & 2364 & 7.798e-7 & 4.497e-7 \\

 & 15000 & 15 & 81 & 1.169e3 & 6152 & 2800 & 8.623e-7 & 4.519e-7 \\
\hline
0.05 & 5000 & 20 & 77  & 4.067e2 & 5957 & 2980 & 8.533e-7 & 2.644e-7 \\

 & 10000 & 20 & 82 & 1.486e3 & 7896 & 3624 & 8.558e-7 & 2.640e-7 \\

 & 15000 & 20 & 88 & 3.467e3 &9867 & 4093  & 8.691e-7 & 2.656e-7 \\
\hline
\end{tabular}
%}
 \end{table}

 We notice that the number of iterations is very robust with respect to the problem dimension $n$, and thus the mesh-size. Unfortunately, this does not lead to a storage demand that is also independent of $n$. The rank of the basis vectors, i.e., the number of columns of the matrices $[\mathcal{V}_{1,1},\ldots,\mathcal{V}_{1,m+1}]$ and 
  $[\mathcal{Z}_{1,1},\ldots,\mathcal{Z}_{1,m}]$ increases with the problem size. This trend is probably inherited from some intrinsic properties of the continuous problem. Indeed, the rank of the computed solution also grows with $n$ suggesting the idea that the rank of the exact solution increases with the problem size as well. Therefore, we are applying low-rank techniques to a problem whose low-rank approximability deteriorates for large $n$ and an increment in the memory requirements of our procedures is thus inevitable. A similar behavior is observed when decreasing the viscosity parameter $\nu$ as well. 
  
  A growth in the rank of the basis vectors determines also a remarkable increment in the computational time as illustrated in Table~\ref{tab1.1}. Indeed, the computational cost of basically all the steps of Algorithm~\ref{LR_FOM}, from the Arnoldi procedure and the low-rank truncations, to the preconditioning phase, depends on the rank of the basis vectors.
  
  We also underline the fact that the true relative residual norm turns out to be always smaller than the normalized computed bound~\eqref{upperbound_res2} validating the reliability of~\eqref{upperbound_res2} as convergence check.
  
  In Figure~\ref{fig:1} (left) we report the normalized bound~\eqref{upperbound_res2} together with the truncation threshold $\varepsilon_{\mathcal{A}}^{(j)}/\|C_1C_2^T\|_F$ for the case $n=5000$ and $\nu=0.5$. We can appreciate how the tolerance for the low-rank truncations increases as the residual norm decreases. As already mentioned, this is a key element to obtain a solution procedure with a feasible storage demand. Moreover,
  in Figure~\ref{fig:1} (right) we document the increment in the rank of the vectors of the preconditioned and unpreconditioned bases as the iterations proceed. We also plot the rank of the unpreconditioned basis we would obtain if no truncations (and no preconditioning steps) were performed, i.e., $4^j$. We can see how we would obtain full-rank basis vectors after very few iterations with consequent impracticable memory requirements of the overall solution process.  
  
   \begin{figure}
  \centering
  \caption{Example~\ref{Ex.1}, $n=5000$, $\nu=0.5$. Left: Normalized bound~\eqref{upperbound_res2} and $\varepsilon_j^{(\mathcal{A})}/\|C_1C_2^T\|_F$ for $j=1,\ldots,9$. Right: Rank of the matrix representing the $j$-th vector of the preconditioned and unpreconditioned basis.} \label{fig:1}
  \begin{minipage}{.5\linewidth}
   	\begin{tikzpicture}
     \begin{semilogyaxis}[width=.95\linewidth, height=.4\textheight,legend pos = south east, xlabel = Iterations, ylabel = Magnitude]
       \addplot+[thick,mark=o] table[x index=0, y index=1]  {data_convdiff.dat};
       \addplot+[thick,mark=square] table[x index=0, y index=2]  {data_convdiff.dat};
       \legend{\eqref{upperbound_res2}$/\|C_1C_2^T\|_F$,$\varepsilon_j^{(\mathcal{A})}/\|C_1C_2^T\|_F$,};
        \end{semilogyaxis}          
   \end{tikzpicture}
   \end{minipage}
   \begin{minipage}{0.48\textwidth} 
   \vspace{0.2cm}
  	\begin{tikzpicture}
    \begin{semilogyaxis}[width=0.95\linewidth, height=.4\textheight, legend pos = north west,
     xlabel = Iterations]
      \addplot+[thick,mark=o] table[x index=0, y index=4]{data_convdiff.dat};
      \addplot+[thick,mark=square] table[x index=0, y index=5]{data_convdiff.dat};
      \addplot+[thick,mark=diamond, color=black] table[x index=0, y index=6]{data_convdiff.dat};
      \legend{$\text{rank}(\mathcal{V}_{1,j})$,$\text{rank}(\mathcal{Z}_{1,j})$,$4^j$};
       \end{semilogyaxis}          
  \end{tikzpicture}
   \end{minipage}
\end{figure}
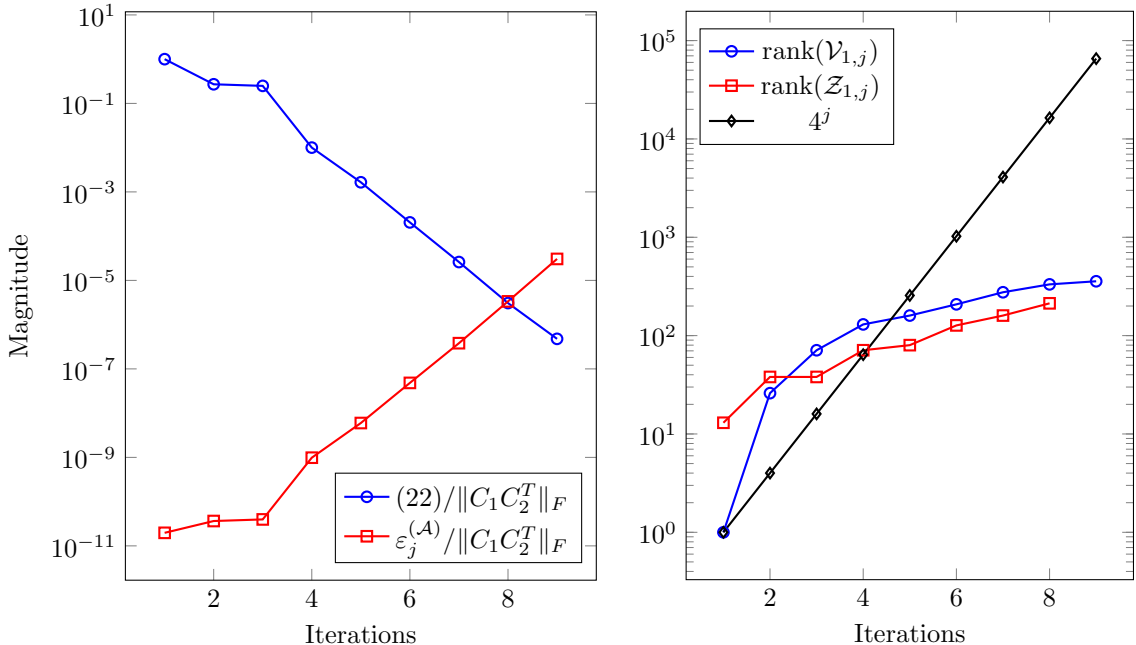
  
  To conclude, in Figure~\ref{Ex.1_Fig.2}, we report the inner product between the last basis vector we have computed and the previous ones, namely we report $\langle \mathcal{V}_{1,9}\mathcal{V}_{2,9}^T,\mathcal{V}_{1,j}\mathcal{V}_{2,j}^T\rangle_F$ for $j=1,\ldots,9$. This numerically confirms that the strategy illustrated in section~\ref{Structured perturbations of the basis} is able to maintain the orthogonality of the basis.
  
  \begin{figure}
  \centering    
  \caption{Example~\ref{Ex.1}, $n=5000$, $\nu=0.5$. $\langle \mathcal{V}_{1,9}\mathcal{V}_{2,9}^T,\mathcal{V}_{1,j}\mathcal{V}_{2,j}^T\rangle_F$ for $j=1,\ldots,9$. {\tt eps} denotes machine precision.}
  \label{Ex.1_Fig.2}
  	\begin{tikzpicture}
    \begin{semilogyaxis}[width=0.8\linewidth, height=.27\textheight,
      legend pos = north west,
      xlabel = Iterations, ylabel = Magnitude ]
      \addplot+[thick] table[x index=0, y index=3]  {data_convdiff.dat};
       \addplot+ [color=black,thick,dashed, mark=none]table[x index=0, y index=7]  {data_convdiff.dat};
        \legend{$\langle\mathcal{V}_{1,9}\mathcal{V}_{2,9}^T\text{,}\mathcal{V}_{1,j}\mathcal{V}_{2,j}^T\rangle_F$,{\tt eps}};
    \end{semilogyaxis}          
  \end{tikzpicture}
  
\end{figure}
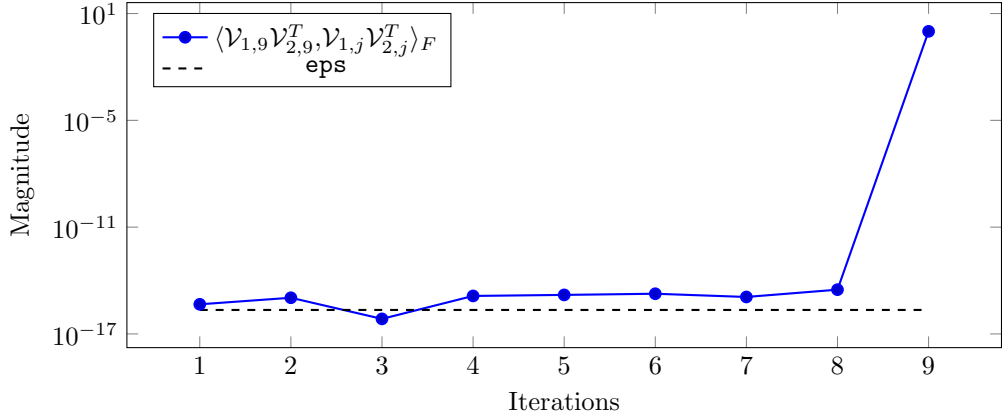

}
\end{num_example}

\begin{num_example}\label{Ex.2}
{\rm
In the second example we consider the algebraic problem stemming from the discretization of stochastic steady-state diffusion equations. In particular, given a sufficiently regular spatial domain $D$ and a sample space $\Omega$ associated with the probability space $(\Omega,\mathcal{F},\mathbb{P})$, we seek an approximation to the function $u:D\times\Omega\rightarrow \mathbb{R}$ which is such that $\mathbb{P}$-almost surely
\begin{equation}\label{Ex.2_eq}
 \begin{array}{rlll}
         -\nabla\cdot(a(x,\omega)\nabla u(x,\omega))&=&f(x),&  \text{in }D,\\
         u(x,\omega)&=&0,&\text{on }\partial D.\\
        \end{array}
\end{equation}
We consider $D=[-1,1]^2$ and we suppose $a$ to be a random field of the form 
$$a(x,\omega)=a_0(x)+\sum_{i=1}^ra_i(x)\sigma_i(\omega),$$
where $\sigma_i:\Omega\rightarrow\Gamma_i\subset\mathbb{R}$ are real-valued independent random variables (RVs).
 %{\color{red} double check the definition of $D$} 

 In our case, $a(x,\omega)$ is a truncated Karhunen-Lo\`eve (KL) expansion 
 \begin{equation}\label{KL}
  a(x,\omega)=\mu(x)+\theta\sum_{i=1}^r\sqrt{\lambda_i}\phi_i(x)\sigma_i(\omega).
 \end{equation}
See, e.g.,~\cite{Lord2014} for more details.
 
The stochastic Galerkin method discussed in, e.g., \cite{Babuska2004,Deb2001,Powell2009,Ullmann2010,Powell2017}, leads to a discrete problem that can be written as a matrix equation of the form 
\begin{equation}\label{eq_stochastic}
 K_0XG_0^T+\sum_{i=1}^rK_iXG_i^T=f_0g_0^T, 
\end{equation}
where $K_i\in\mathbb{R}^{n_x\times n_x}$, $G_i\in\mathbb{R}^{n_\sigma\times n_\sigma}$, and $f_0
\in\mathbb{R}^{n_x}$, $g_0
\in\mathbb{R}^{n_\sigma}$. See, e.g.,~\cite{Powell2009,Powell2017}.

%\in\mathbb{R}^{n_\sigma}$. %See, e.g.,~\cite{Powell2017}.
%>>>>>>> 4d907070e3996424b211d9443a11bec0816494d4

We solve equation~\eqref{eq_stochastic} by LR-GMRES and the following operators
$$
\begin{array}{rrll}
\mathcal{P}_{\mathtt{mean}}:& \mathbb{R}^{n_x\times n_\sigma}&\rightarrow& \mathbb{R}^{n_x\times n_\sigma}\\
&X&\mapsto&K_0X,
\end{array}\quad
\begin{array}{rrll}
\mathcal{P}_{\mathtt{Ullmann}}:& \mathbb{R}^{n_x\times n_\sigma}&\rightarrow& \mathbb{R}^{n_x\times n_\sigma}\\
&X&\mapsto&K_0X\widebar G^T,\quad \widebar G:=\sum_{i=0}^r\frac{\text{trace}(K_i^TK_0)}{\text{trace}(K_0^TK_0)}G_i,
\end{array}
$$
are selected as preconditioners. $\mathcal{P}_{\mathtt{mean}}$ is usually referred to as mean-based preconditioner, see, e.g.,~\cite{Powell2017, Powell2009} and the references therein, while Ullmann proposed $\mathcal{P}_{\mathtt{Ullmann}}$ in~\cite{Ullmann2010}.

Both $\mathcal{P}_{\mathtt{mean}}$ and $\mathcal{P}_{\mathtt{Ullmann}}$ are very well-suited for our framework as their application amount to the solution of a couple of linear systems so that the rank of the current basis vector does not increase. See the discussion in section~\ref{Preconditioning}. Moreover, supposing that these linear systems can be solved exactly by, e.g., a sparse direct solver, there is no need to employ flexible GMRES so that only one basis has to be stored. In particular, in all our tests, we precompute once and for all the LU factors of the matrices\footnote{The computational time of such decompositions is always included in the reported results.} which define the selected preconditioner so that only triangular systems are solved during the LR-GMRES iterations.

We %consider two 
%which differ in the problem dimensions $n_x$, $n_\sigma$, and in the number of linear terms $r+1$. The two sets of data we consider have been 
generate instances of~\eqref{eq_stochastic} with the help of the S-IFISS\footnote{Available at \tt{https://personalpages.manchester.ac.uk/staff/david.silvester/ifiss/sifiss.html}} package version 1.04; see \cite{Silvester2015}. 
The S-IFISS routine \texttt{stoch\_diff\_testproblem\_pc} is executed to generate two instances of~\eqref{eq_stochastic}. The first equation ({\tt Data 1}) is obtained by using a spatial discretization with $2^7$ points in each dimension, $r=2$ RVs in~\eqref{KL} which are approximated by polynomial chaos expansions of length $\ell=100$ leading to
$n_x=16129$, $n_\sigma=5151$, and $r+1=3$. 
%
%($2^7$ grid points in each dimension, $r=2$ random variables, polynomial chaos expansions of length 100).
The second instance ({\tt Data 2}) was generated with $2^8$ grid points, $r=5$, and chaos expansions of length $\ell=10$ resulting in $n_x=65025$, $n_\sigma=3003$, and $r+1=6$.
%($2^8$ grid points in each dimension, $r=5$ RVs, PCE of length 10).
%.
%{\color{red} There is a lot of data involved in setting up these examples... Maybe we don't need that much detail here..}

%In particular, in the first equation ({\tt Data 1}) we have $n_x=16129$, $n_\sigma=5151$ and $r+1=3$, whereas in the second one ({\tt Data 2}) $n_x=65025$, $n_\sigma=3003$ and $r+1=6$.

 \begin{table}[!ht]
 \centering
% {\small
 \caption{Example~\ref{Ex.2}. Results of preconditioned LR-GMRES applied to different test problems. {\tt Data 1}: $n_x=16129$, $n_\sigma=5151$, $r+1=3$, {\tt Data 2}: $n_x=65025$, $n_\sigma=3003$, $r+1=6$.} \label{tab2.1}
 \vspace{0.3cm}
 
  \setlength\tabcolsep{10pt}
 \begin{tabular}{| r | r r r | r r | r |}
 \hline
 & & & & \multicolumn{2}{c|}{Conv. Checks} & \\
  Prec.  & Its  & 
  rank($S_1S_2^T$) & Mem. &~\eqref{upperbound_res2}$/\|C_1C_2^T\|_F$ & Real Res. & Time (s) \\
 
 \hline
 \multicolumn{7}{|c|}{{\tt Data 1}}\\
\hline  
 $\mathcal{P}_{\mathtt{Ullmann}}$ &  9  & 44 & 220 & 3.703e-7 &  3.551e-7& 1.204e1\\
 $\mathcal{P}_{\mathtt{mean}}$ &  13  & 64 & 507 & 7.636e-7 & 7.369e-7 & 2.521e1\\
 \hline
 \multicolumn{7}{|c|}{{\tt Data 2}}\\
\hline  
 $\mathcal{P}_{\mathtt{Ullmann}}$ &  15  & 791 & 10266 & 5.611e-7 & 5.359e-7&  8.847e4\\
 $\mathcal{P}_{\mathtt{mean}}$ &  20  & 806 & 14912 & 8.118e-7 & 7.703e-7& 1.626e5\\
 \hline 
\end{tabular}
 \end{table}
 
Table~\ref{tab2.1} summarizes the results and apparently problem {\tt Data 2} is much more challenging than {\tt Data 1}. This is meanly due to the number of terms in~\eqref{eq_stochastic}. Indeed, the effectiveness of the preconditioners may deteriorate as $r$ increases even though the actual capability of $\mathcal{P}_{\mathtt{mean}}$ and $\mathcal{P}_{\mathtt{Ullmann}}$ in reducing the iteration count is related to the coefficients of the KL expansion~\eqref{KL}. See, e.g.,~\cite[Theorem 3.8]{Powell2009} and~\cite[Corollary 5.4]{Ullmann2010}.
Moreover, $r+1$ terms are involved in the products in line~\ref{alg_line_product} of Algorithm~\ref{LR_FOM} and a sizable $r$ leads, in general, to a faster growth in the rank of the basis vectors so that a larger number of columns are retained during the truncation step in line~\ref{algo_firsttrunc}. As a result, the computational cost of our iterative scheme increases as well leading to a rather time consuming routine.

If the discrete operator stemming from the discretization of~\eqref{Ex.2_eq} is well posed, then it is also symmetric positive definite and the CG method can be employed in the solution process. See, e.g.,~\cite[Section 3]{Powell2009}. We thus try to apply the (preconditioned) low-rank variant of CG (LR-CG) to the matrix equation~\eqref{eq_stochastic}. To this end, we adopt the LR-CG implementation proposed in~\cite{Benner2013}. With the notation of~\cite[Algorithm 1]{Benner2013} we truncate all the iterates $X_{k+1}$, $R_{k+1}$, $P_{k+1}$ and $Q_{k+1}$. In particular, the threshold for the truncation of $X_{k+1}$ is set to 
$10^{-12}$ while the value on the right-hand side of~\eqref{eq_CERFACS} is used at the $k$-th LR-CG iteration for the low-rank truncation of all the other iterates. We want to point out that in the LR-CG implementation proposed in~\cite{Benner2013}, the residual matrix $R_{k+1}$ is explicitly calculated by means of the current approximate solution $X_{k+1}$. We compute the residual norm before truncating $R_{k+1}$ so that what we are actually evaluating is the true residual norm and not an upper bound thereof.

The results are collected in Table~\ref{tab2.2} where the column ``Mem.'' reports the maximum number of columns
that had to be stored in the low-rank factors of all the iterates $X_{k+1}$, $R_{k+1}$, $P_{k+1}$, $Q_{k+1}$, and $Z_{k+1}$.

{\renewcommand{\arraystretch}{1.2}%
 \begin{table}[!ht]
 \centering
 \caption{Example~\ref{Ex.2}. Results of preconditioned LR-CG applied to different test problems. {\tt Data 1}: $n_x=16129$, $n_\sigma=5151$, $r+1=3$, {\tt Data 2}: $n_x=65025$, $n_\sigma=3003$, $r+1=6$.} \label{tab2.2}
 \vspace{0.3cm}
 
  \setlength\tabcolsep{10pt}
 \begin{tabular}{| r | r r r  r  r |}
 \hline
  Prec. & Its  & 
  rank($S_1S_2^T$) & Mem.  & Real Res. & Time (s) \\
 
 \hline
 \multicolumn{6}{|c|}{{\tt Data 1}}\\
\hline  
 $\mathcal{P}_{\mathtt{Ullmann}}$ &  11  & 41 & 234 & 9.517e-7 & 1.921e0\\
 $\mathcal{P}_{\mathtt{mean}}$ &  19  & 52 & 288 & 9.629e-7 & 3.369e0\\
 \hline
 \multicolumn{6}{|c|}{{\tt Data 2}}\\
\hline  
 $\mathcal{P}_{\mathtt{Ullmann}}$ &  46  & 483 & 4404 & 9.976e-7 &  9.642e2\\
 $\mathcal{P}_{\mathtt{mean}}$ &  67  & 450 & 4096 & 9.981e-7 & 1.325e3\\
 \hline 
\end{tabular}
 \end{table}
}

 Except for {\tt Data 1} with $\mathcal{P}_{\mathtt{Ullmann}}$ as a preconditioner where LR-GMRES and LR-CG show similar results especially in terms of memory requirements, LR-CG allows for a much lower storage demand with a consequent reduction in the total computational efforts while achieving the prescribed accuracy.
 However, for {\tt Data 2}, LR-CG requires a rather large number of iterations to converge regardless of the adopted preconditioner. This is due to a very small reduction of the residual norm, almost a stagnation, from one iteration to the following one we observe in the final stage of the algorithm. See Figure~\ref{fig:2} (left). This issue may be fixed by employing a more robust, possibly more conservative, threshold for the low-rank truncations. Alternatively, a condition of the form $\|X_k-X_{k+1}\|_F\leq \varepsilon$ can be included in the convergence check as proposed in~\cite{Powell2017}.
 
 We conclude by mentioning a somehow surprising behavior of LR-CG. In particular, in the first iterations the rank of all the iterates increases as expected, while it starts decreasing from a certain $\widebar k$ on until it reaches an almost constant value. See Figure~\ref{fig:2} (right). This trend allows for a feasible storage demand also when many iterations are performed as for {\tt Data 2}.
 We think that such a phenomenon deserves further studies.
 %\pnote{Which iterates are shown in Figure~\ref{fig:2}? I assume $X_k$?!}
  
   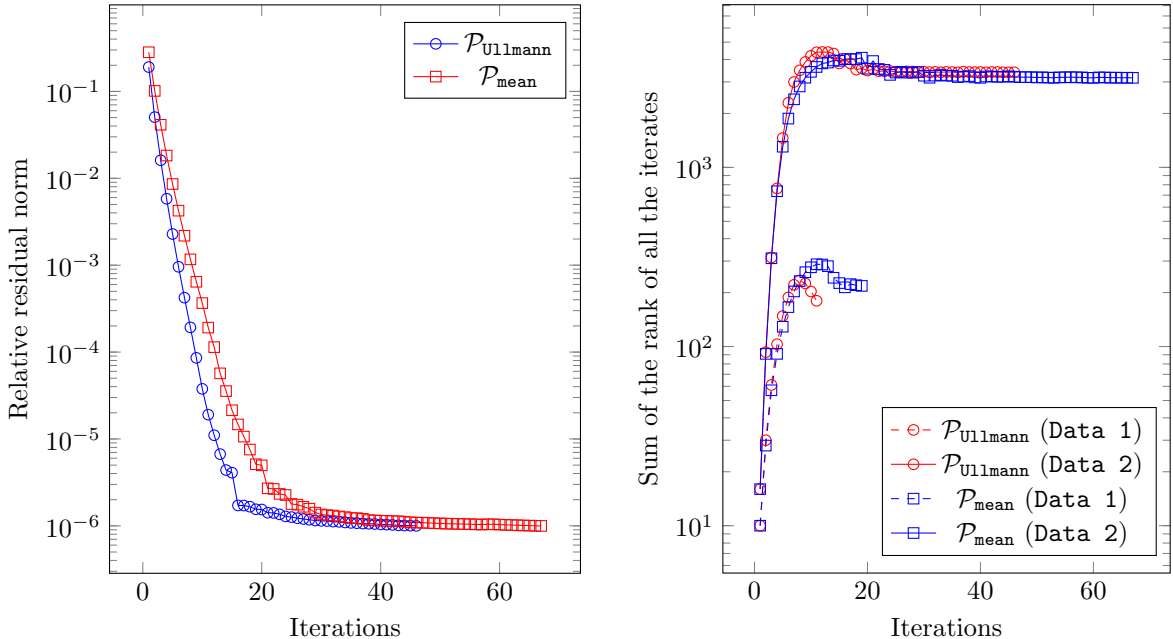
\begin{figure}
  \centering
  \caption{Example~\ref{Ex.2}. Left: LR-CG relative residual norm for {\tt Data 2}. Right: Sum of the rank of all the LR-CG iterates $X_{k+1}$, $R_{k+1}$, $P_{k+1}$, $Q_{k+1}$, and $Z_{k+1}$ as the iterations proceed.} \label{fig:2}
  \vspace{0.2cm}
  
  \begin{minipage}{.5\linewidth}
   	\begin{tikzpicture}
     \begin{semilogyaxis}[width=.95\linewidth, height=.4\textheight,legend pos = north east, xlabel = Iterations, ylabel = Relative residual norm]
       \addplot+[mark=o] table[x index=0, y index=1]  {data_stochastlarge_CG_Ullmann.dat};
       \addplot+[mark=square] table[x index=0, y index=1]  {data_stochastlarge_CG_meanbased.dat};
       \legend{$\mathcal{P}_{\mathtt{Ullmann}}$,$\mathcal{P}_{\mathtt{mean}}$};
        \end{semilogyaxis}          
   \end{tikzpicture}
   \end{minipage}
   \begin{minipage}{0.48\textwidth} 
   %\vspace{0.2cm}
  	
  	\begin{tikzpicture}
    \begin{semilogyaxis}[width=0.95\linewidth, height=.4\textheight, legend pos = south east,
     xlabel = Iterations, ylabel = Sum of the rank of all the iterates]
      \addplot+[ mark= o,dashed, color = red, mark options={solid}] table[x index=0, y index=2]{data_stochastsmall_CG_Ullmann.dat};      
      \addplot+[ mark=o, color = red] table[x index=0, y index=2]{data_stochastlarge_CG_Ullmann.dat};
      \addplot+[ mark= square,dashed, color = blue, mark options={solid}] table[x index=0, y index=2]{data_stochastsmall_CG_meanbased.dat};
      \addplot+[ mark=square, color = blue] table[x index=0, y index=2]{data_stochastlarge_CG_meanbased.dat};
      
 \legend{$\mathcal{P}_{\mathtt{Ullmann}}$ ({\tt Data 1}), $\mathcal{P}_{\mathtt{Ullmann}}$ ({\tt Data 2}),$\mathcal{P}_{\mathtt{mean}}$ ({\tt Data 1}), $\mathcal{P}_{\mathtt{mean}}$ ({\tt Data 2})};
 \end{semilogyaxis}          
  \end{tikzpicture}
   \end{minipage}
\end{figure}

}
\end{num_example}

%%%%%%%%%%%%%%%%%%%%%%%%%%%%%%%%%%%%%%%%%%%%%%%%%%
\section{Conclusions}\label{Conclusions}
Low-rank Krylov methods are one of the few options for solving general linear matrix equations of the form~\eqref{eq_main}, especially for large problem dimensions. An important step of these procedures consist in truncating the rank of the basis vectors to maintain a feasible storage demand of the overall solution process. In principle, such truncations can severely impact on the converge of the adopted Krylov routine.

In this paper we have shown how to perform the low-rank truncations in order to maintain the convergence of the selected Krylov procedure.
In particular, our analysis points out that not only the thresholds employed for the truncations are important, but also the actual procedure adopted for the low-rank truncations plays a fundamental role. Indeed, such a routine must be able to preserve the orthogonality of the computed basis.

%%%%%%%%%%%%%%%%%%%%%%%%%%%%%%%%%%%%%%%%%%%%%%%%%%%%%%%%%%%%%%%%
\section*{Acknowledgments}
 The first author is a member of the Italian INdAM Research group GNCS. Part of this work was carried out while the second author was affiliated with the Max Planck Institute for Dynamics of Complex Technical Systems in Magdeburg, Germany.

\bibliography{LowRankKrylov}

\end{document}

%% file: generic.tex
\usepackage{amsmath,amssymb}
\usepackage{cite}
\usepackage{mathtools}
\usepackage{tikz}
\usepackage{pgfplots}
\usepackage{pgfplotstable}
\usepackage{geometry}
\usepackage{color}
\usepackage{booktabs}
\usepackage{framed}
\pgfplotsset{compat=1.9}

\pgfplotstableset{
	every head row/.style={before row=\toprule,after row=\midrule},
	clear infinite
}

\usepackage{amsopn}

\renewcommand{\leq}{\leqslant}
\renewcommand{\geq}{\geqslant}
\renewcommand{\tilde}{\widetilde}

%% file: LowRankKrylovMethods_newversion.bbl
\begin{thebibliography}{10}

\bibitem{Antoulas.05}
{\sc A.~C. Antoulas}, {\em Approximation of large-scale dynamical systems},
  vol.~6 of Advances in Design and Control, Society for Industrial and Applied
  Mathematics (SIAM), Philadelphia, PA, 2005.

\bibitem{Babuska2004}
{\sc I.~Babu\v{s}ka, R.~Tempone, and G.~E. Zouraris}, {\em Galerkin finite
  element approximations of stochastic elliptic partial differential
  equations}, SIAM J. Numer. Anal., 42 (2004), pp.~800--825.

\bibitem{BagRei05}
{\sc J.~Baglama and L.~Reichel}, {\em Augmented implicitly restarted {L}anczos
  bidiagonalization methods}, {SIAM} J. Sci. Comput., 27 (2005), pp.~19--42.

\bibitem{Baker2015}
{\sc J.~Baker, M.~Embree, and J.~Sabino}, {\em Fast singular value decay for
  {L}yapunov solutions with nonnormal coefficients}, SIAM J. Matrix Anal.
  Appl., 36 (2015), pp.~656--668.

\bibitem{Baumann2018}
{\sc M.~Baumann, R.~Astudillo, Y.~Qiu, E.~Y. M.~Ang, M.~B. van Gijzen, and
  R.-{\'E}. Plessix}, {\em An {MSSS}-preconditioned matrix equation approach
  for the time-harmonic elastic wave equation at multiple frequencies},
  Computational Geosciences, 22 (2018), pp.~43--61.

\bibitem{Baur2008}
{\sc U.~Baur}, {\em Low rank solution of data-sparse {S}ylvester equations},
  Numer. Linear Algebra Appl., 15 (2008), pp.~837--851.

\bibitem{Baur2006}
{\sc U.~Baur and P.~Benner}, {\em Factorized solution of {L}yapunov equations
  based on hierarchical matrix arithmetic}, Computing, 78 (2006), pp.~211--234.

\bibitem{Benner2013}
{\sc P.~Benner and T.~Breiten}, {\em Low rank methods for a class of
  generalized {L}yapunov equations and related issues}, Numer. Math., 124
  (2013), pp.~441--470.

\bibitem{Benner2011}
{\sc P.~Benner and T.~Damm}, {\em Lyapunov equations, energy functionals, and
  model order reduction of bilinear and stochastic systems}, SIAM J. Control
  Optim., 49 (2011), pp.~686--711.

\bibitem{Benner2009}
{\sc P.~Benner, R.-C. Li, and N.~Truhar}, {\em On the {ADI} method for
  {S}ylvester equations}, J. Comput. Appl. Math., 233 (2009), pp.~1035--1045.

\bibitem{Benner2017}
{\sc P.~Benner, M.~Ohlberger, A.~Cohen, and K.~Willcox}, {\em Model Reduction
  and Approximation}, Society for Industrial and Applied Mathematics,
  Philadelphia, PA, 2017.

\bibitem{BenOnwSto15}
{\sc P.~Benner, A.~Onwunta, and M.~Stoll}, {\em Low-rank solution of unsteady
  diffusion equations with stochastic coefficients}, SIAM/ASA Journal on
  Uncertainty Quantification, 3 (2015), pp.~622--649.

\bibitem{Bouras2005}
{\sc A.~Bouras and V.~Frayss\'{e}}, {\em Inexact matrix-vector products in
  {K}rylov methods for solving linear systems: a relaxation strategy}, SIAM J.
  Matrix Anal. Appl., 26 (2005), pp.~660--678.

\bibitem{Breiten2016}
{\sc T.~Breiten, V.~Simoncini, and M.~Stoll}, {\em Low-rank solvers for
  fractional differential equations}, Electron. Trans. Numer. Anal., 45 (2016),
  pp.~107--132.

\bibitem{Che2019}
{\sc M.~Che and Y.~Wei}, {\em Randomized algorithms for the approximations of
  {T}ucker and the tensor train decompositions}, Advances in Computational
  Mathematics, 45 (2019), pp.~395--428.

\bibitem{Damm2008}
{\sc T.~Damm}, {\em Direct methods and {ADI}-preconditioned {K}rylov subspace
  methods for generalized {L}yapunov equations}, Numer. Linear Algebra Appl.,
  15 (2008), pp.~853--871.

\bibitem{Deb2001}
{\sc M.~K. Deb, I.~M. Babu\v{s}ka, and J.~T. Oden}, {\em Solution of stochastic
  partial differential equations using {G}alerkin finite element techniques},
  Comput. Methods Appl. Mech. Engrg., 190 (2001), pp.~6359--6372.

\bibitem{Dolgov2013}
{\sc S.~V. Dolgov}, {\em T{T}-{GMRES}: solution to a linear system in the
  structured tensor format}, Russian J. Numer. Anal. Math. Modelling, 28
  (2013), pp.~149--172.

\bibitem{Druskin.Simoncini.11}
{\sc V.~Druskin and V.~Simoncini}, {\em Adaptive rational {K}rylov subspaces
  for large-scale dynamical systems}, Systems Control Lett., 60 (2011),
  pp.~546--560.

\bibitem{Freitag2018}
{\sc M.~A. Freitag and D.~L.~H. Green}, {\em A low-rank approach to the
  solution of weak constraint variational data assimilation problems}, J.
  Comput. Phys., 357 (2018), pp.~263--281.

\bibitem{Freund1991}
{\sc R.~W. Freund and N.~M. Nachtigal}, {\em Q{MR}: a quasi-minimal residual
  method for non-{H}ermitian linear systems}, Numer. Math., 60 (1991),
  pp.~315--339.

\bibitem{Giraud2005}
{\sc L.~Giraud, J.~Langou, M.~Rozlo\v{z}n\'{\i}k, and J.~van~den Eshof}, {\em
  Rounding error analysis of the classical {G}ram-{S}chmidt orthogonalization
  process}, Numer. Math., 101 (2005), pp.~87--100.

\bibitem{Giraud2005a}
{\sc L.~Giraud, J.~Langou, and M.~Rozloznik}, {\em The loss of orthogonality in
  the {G}ram-{S}chmidt orthogonalization process}, Comput. Math. Appl., 50
  (2005), pp.~1069--1075.

\bibitem{Gutknecht2006}
{\sc M.~H. Gutknecht}, {\em {K}rylov subspace algorithms for systems with
  multiple right hand sides: an introduction}, in Modern mathematical models,
  methods and algorithms for real world systems, A.~Siddiqi, I.~Duff, and
  O.~Christensen, eds., Anshan {L}td, 2007.
\newblock Available at {\tt http://www.sam.math.ethz.ch/\textasciitilde
  mhg/pub/delhipap.pdf}.

\bibitem{Gue10}
{\sc S.~G{\"u}ttel}, {\em Rational {K}rylov methods for operator functions},
  PhD thesis, Technische Universit{\"a}t Bergakademie Freiberg, Germany, 2010.
\newblock Available online from the Qucosa server.

\bibitem{HalMT11}
{\sc N.~Halko, P.~Martinsson, and J.~Tropp}, {\em Finding structure with
  randomness: Probabilistic algorithms for constructing approximate matrix
  decompositions}, SIAM Review, 53 (2011), pp.~217--288.

\bibitem{Hestenes1952}
{\sc M.~R. Hestenes and E.~Stiefel}, {\em Methods of conjugate gradients for
  solving linear systems}, J. Research Nat. Bur. Standards, 49 (1952),
  pp.~409--436 (1953).

\bibitem{Hoc01}
{\sc M.~E. Hochstenbach}, {\em A {J}acobi--{D}avidson type {SVD} method},
  {SIAM} J. Sci. Comput., 23 (2001), pp.~606--628.

\bibitem{Jarlebring2018}
{\sc E.~Jarlebring, G.~Mele, D.~Palitta, and E.~Ringh}, {\em Krylov methods for
  low-rank commuting generalized sylvester equations}, Numerical Linear Algebra
  with Applications, 25 (2018).
\newblock e2176.

\bibitem{Kandler2019}
{\sc U.~Kandler}, {\em Inexact methods for the solution of large scale
  {H}ermitian eigenvalue problems}, PhD thesis, Technische Universit\"{a}t
  Berlin, 2019.

\bibitem{KreP17}
{\sc D.~Kressner and L.~Peri\v{s}a}, {\em Recompression of {H}adamard products
  of tensors in {T}ucker format}, SIAM Journal on Scientific Computing, 39
  (2017), pp.~A1879--A1902.

\bibitem{Kressner2015}
{\sc D.~Kressner and P.~Sirkovi\'{c}}, {\em Truncated low-rank methods for
  solving general linear matrix equations}, Numer. Linear Algebra Appl., 22
  (2015), pp.~564--583.

\bibitem{Kressner2016}
{\sc D.~Kressner, M.~Steinlechner, and B.~Vandereycken}, {\em Preconditioned
  low-rank {R}iemannian optimization for linear systems with tensor product
  structure}, SIAM J. Sci. Comput., 38 (2016), pp.~A2018--A2044.

\bibitem{Kressner2009/10}
{\sc D.~Kressner and C.~Tobler}, {\em Krylov subspace methods for linear
  systems with tensor product structure}, SIAM J. Matrix Anal. Appl., 31
  (2009/10), pp.~1688--1714.

\bibitem{Kressner2011}
\leavevmode\vrule height 2pt depth -1.6pt width 23pt, {\em Low-rank tensor
  {K}rylov subspace methods for parametrized linear systems}, SIAM J. Matrix
  Anal. Appl., 32 (2011), pp.~1288--1316.

\bibitem{Kuerschner2018}
{\sc P.~K\"{u}rschner, S.~Dolgov, K.~D. Harris, and P.~Benner}, {\em Greedy
  low-rank algorithm for spatial connectome regression}, J. Math. Neurosci., 9
  (2019).

\bibitem{Lar98}
{\sc R.~Larsen}, {\em Lanczos bidiagonalization with partial
  reorthogonalization}, DAIMI Report Series, 27 (1998).

\bibitem{Li2004}
{\sc J.-R. Li and J.~White}, {\em Low-rank solution of {L}yapunov equations},
  SIAM Rev., 46 (2004), pp.~693--713.

\bibitem{LieStr12}
{\sc J.~Liesen and Z.~Strakos}, {\em Krylov subspace methods: Principles and
  analysis}, {Oxford University Press}, 2012.

\bibitem{LiuMW15}
{\sc Q.~Liu, R.~B. Morgan, and W.~Wilcox}, {\em {Polynomial preconditioned
  GMRES and GMRES-DR}}, SIAM J. Sci. Comput., 37 (2015), pp.~S407--S428.

\bibitem{Lord2014}
{\sc G.~J. Lord, C.~E. Powell, and T.~Shardlow}, {\em An introduction to
  computational stochastic {PDE}s}, Cambridge Texts in Applied Mathematics,
  Cambridge University Press, New York, 2014.

\bibitem{MATLAB}
{\sc {MATLAB}}, {\em version 9.3.0 (R2017b)}, The MathWorks Inc., Natick,
  Massachusetts, 2017.

\bibitem{Onw16}
{\sc A.~Onwunta}, {\em Low-rank iterative solvers for stochastic {G}alerkin
  linear systems}, {D}issertation, Otto-von-Guericke-Universit{\"a}t,
  Magdeburg, Germany, 2016.

\bibitem{Paige1975}
{\sc C.~C. Paige and M.~A. Saunders}, {\em Solutions of sparse indefinite
  systems of linear equations}, SIAM J. Numer. Anal., 12 (1975), pp.~617--629.

\bibitem{Palitta2016}
{\sc D.~Palitta and V.~Simoncini}, {\em Matrix-equation-based strategies for
  convection-diffusion equations}, BIT, 56 (2016), pp.~751--776.

\bibitem{Penzl2000}
{\sc T.~Penzl}, {\em Eigenvalue decay bounds for solutions of {L}yapunov
  equations: the symmetric case}, Systems Control Lett., 40 (2000),
  pp.~139--144.

\bibitem{Powell2009}
{\sc C.~E. Powell and H.~C. Elman}, {\em Block-diagonal preconditioning for
  spectral stochastic finite-element systems}, IMA J. Numer. Anal., 29 (2009),
  pp.~350--375.

\bibitem{Powell2017}
{\sc C.~E. Powell, D.~Silvester, and V.~Simoncini}, {\em An efficient reduced
  basis solver for stochastic {G}alerkin matrix equations}, SIAM J. Sci.
  Comput., 39 (2017), pp.~A141--A163.

\bibitem{Ringh2018}
{\sc E.~Ringh, G.~Mele, J.~Karlsson, and E.~Jarlebring}, {\em Sylvester-based
  preconditioning for the waveguide eigenvalue problem}, Linear Algebra Appl.,
  542 (2018), pp.~441--463.

\bibitem{Saad1993}
{\sc Y.~Saad}, {\em A flexible inner-outer preconditioned {GMRES} algorithm},
  SIAM J. Sci. Comput., 14 (1993), pp.~461--469.

\bibitem{Saad2003}
\leavevmode\vrule height 2pt depth -1.6pt width 23pt, {\em Iterative methods
  for sparse linear systems}, {SIAM}, Society for Industrial and Applied
  Mathematics, Philadelphia, PA, 2nd~ed., 2003.

\bibitem{Saad1986}
{\sc Y.~Saad and M.~H. Schultz}, {\em G{MRES}: a generalized minimal residual
  algorithm for solving nonsymmetric linear systems}, SIAM J. Sci. Statist.
  Comput., 7 (1986), pp.~856--869.

\bibitem{Shank2015}
{\sc S.~D. Shank, V.~Simoncini, and D.~B. Szyld}, {\em Efficient low-rank
  solution of generalized {L}yapunov equations}, Numer. Math., 134 (2016),
  pp.~327--342.

\bibitem{Silvester2015}
{\sc D.~J. Silvester, A.~Bespalov, and C.~E. Powell}, {\em {S-IFISS} version
  1.04}, 2017.

\bibitem{Simoncini2007}
{\sc V.~Simoncini}, {\em A new iterative method for solving large-scale
  {L}yapunov matrix equations}, SIAM J. Sci. Comput., 29 (2007),
  pp.~1268--1288.

\bibitem{Simoncini2016}
\leavevmode\vrule height 2pt depth -1.6pt width 23pt, {\em Computational
  methods for linear matrix equations}, SIAM Rev., 58 (2016), pp.~377--441.

\bibitem{Simoncini2002}
{\sc V.~Simoncini and D.~B. Szyld}, {\em Flexible inner-outer {K}rylov subspace
  methods}, SIAM J. Numer. Anal., 40 (2002), pp.~2219--2239 (2003).

\bibitem{Simoncini2003}
\leavevmode\vrule height 2pt depth -1.6pt width 23pt, {\em Theory of inexact
  {K}rylov subspace methods and applications to scientific computing}, SIAM J.
  Sci. Comput., 25 (2003), pp.~454--477.

\bibitem{Simoncini2007a}
\leavevmode\vrule height 2pt depth -1.6pt width 23pt, {\em Recent computational
  developments in {K}rylov subspace methods for linear systems}, Numer. Linear
  Algebra Appl., 14 (2007), pp.~1--59.

\bibitem{Sto12}
{\sc M.~Stoll}, {\em A {Krylov-Schur} approach to the truncated {SVD}}, Linear
  Algebra Appl., 436 (2012), pp.~2795--2806.

\bibitem{Stoll2015}
{\sc M.~Stoll and T.~Breiten}, {\em A low-rank in time approach to
  {PDE}-constrained optimization}, SIAM J. Sci. Comput., 37 (2015),
  pp.~B1--B29.

\bibitem{Ullmann2010}
{\sc E.~Ullmann}, {\em A {K}ronecker product preconditioner for stochastic
  {G}alerkin finite element discretizations}, SIAM J. Sci. Comput., 32 (2010),
  pp.~923--946.

\bibitem{Eshof2004}
{\sc J.~van~den Eshof and G.~L.~G. Sleijpen}, {\em Inexact {K}rylov subspace
  methods for linear systems}, SIAM J. Matrix Anal. Appl., 26 (2004),
  pp.~125--153.

\bibitem{Vorst1992}
{\sc H.~A. van~der Vorst}, {\em Bi-{CGSTAB}: a fast and smoothly converging
  variant of {B}i-{CG} for the solution of nonsymmetric linear systems}, SIAM
  J. Sci. Statist. Comput., 13 (1992), pp.~631--644.

\bibitem{Dooren1991}
{\sc P.~M. Van~Dooren}, {\em Structured linear algebra problems in digital
  signal processing}, in Numerical linear algebra, digital signal processing
  and parallel algorithms ({L}euven, 1988), vol.~70 of NATO Adv. Sci. Inst.
  Ser. F Comput. Systems Sci., Springer, Berlin, 1991, pp.~361--384.

\bibitem{Gij95}
{\sc M.~B. van Gijzen}, {\em {A polynomial preconditioner for the {GMRES}
  algorithm}}, J. Comput. Appl. Math., 59 (1995), pp.~91--107.

\bibitem{Vandereycken2010}
{\sc B.~Vandereycken and S.~Vandewalle}, {\em A {R}iemannian optimization
  approach for computing low-rank solutions of {L}yapunov equations}, SIAM J.
  Matrix Anal. Appl., 31 (2010), pp.~2553--2579.

\bibitem{WeiBR19}
{\sc R.~{Weinhandl}, P.~{Benner}, and T.~{Richter}}, {\em {Low-rank Linear
  Fluid-structure Interaction Discretizations}}, arXiv e-prints,  (2019).
\newblock ArXiv: 1905.11000.

\end{thebibliography}
